\documentclass[12pt,psfig,reqno]{amsart}
\usepackage{mathrsfs}
\usepackage{hyperref,latexsym}
\usepackage{txfonts}
\usepackage{amscd}
\usepackage{cite}
\newif\ifdraft\draftfalse
\newif\ifcompact\compactfalse

\compacttrue 

\ifdraft 
    \usepackage{ctex}
   \def\draft#1{\red{(\textbf{Note: }\brown{#1})}}
   \def\draftt#1{\textbf{\red{Note: (草稿解释图片)}}#1}
    \renewenvironment{proof}[1][Proof]
    	{\par\noindent\textit{#1}.\,}{$\hfill\Box$}
    \else  
        \def\draft#1{\relax}\def\draftt#1{\relax}
\fi
\ifcompact 
    
    \else 
        
\fi

\usepackage{epsfig}
\usepackage{verbatim}
\usepackage{mathdots}
\usepackage{amssymb}
\usepackage{amsfonts}
\usepackage{amsbsy}
\usepackage{caption}
\usepackage{subcaption}
\usepackage{graphicx}
\usepackage{float}
\usepackage{epstopdf}
\usepackage[usenames]{color} 
\usepackage{cancel}

 \numberwithin{equation}{section}

\newtheorem{prop}{Proposition}[section]
\newtheorem{lem}[prop]{Lemma}

\newtheorem{defi}{Definition}[section]
\newtheorem{cor}[prop]{Corollary}
\newtheorem{thm}[prop]{Theorem}

\newtheorem{exam}[prop]{Example}
\newtheorem{rem}[prop]{Remark}

\begin{document}
\baselineskip=17pt

\title[Existence of $L^q$-dimension and entropy dimension]
{Existence of $L^q$-dimension and entropy dimension of self-conformal measures on Riemannian manifolds}
\date{\today}
\author[S.-M. Ngai]{Sze-Man Ngai}
\address{Key Laboratory of High Performance Computing and Stochastic Information
		Processing (HPCSIP) (Ministry of Education of China), College of Mathematics and Statistics, Hunan Normal University,
	Changsha, Hunan 410081, China, and Department of Mathematical Sciences\\ Georgia Southern
	University\\ Statesboro, GA 30460-8093, USA.}
\email{smngai@georgiasouthern.edu}
\author[Y. Xu]{Yangyang Xu*}
\address{Key Laboratory of High Performance Computing and Stochastic Information
		Processing (HPCSIP) (Ministry of Education of China), College of Mathematics and Statistics, Hunan Normal University,
	Changsha, Hunan 410081, China.} \email{yyxumath@163.com}

\thanks{* Corresponding author.}
\thanks{The authors are supported in part by the National Natural Science Foundation of China, grant 11771136, and Construct Program of the Key Discipline in Hunan Province. The first author is also supported in part by a Faculty Research Scholarly Pursuit Funding from Georgia Southern University.}

\begin{abstract} Peres and Solomyak proved that on $\mathbb R^n$, the limits defining the $L^q$-dimension for any $q\in(0,\infty)\setminus\{1\}$, and the entropy dimension of a self-conformal measure exist, without assuming any separation condition. By introducing the notions of heavy maximal packings and partitions, we prove that on a doubling metric space the $L^q$-dimension, $q\in(0,\infty)\setminus\{1\}$, is equivalent to the generalized dimension. We also generalize the result on the existence of the $L^q$-dimension to self-conformal measures on complete Riemannian manifolds with the doubling property. In particular, these results hold for complete Riemannian manifolds with nonnegative Ricci curvature. Moreover, by assuming that the measure is doubling, we extend the result on the existence of the entropy dimension to self-conformal measures on complete Riemannian manifolds.
\end{abstract}

\maketitle

\section{\bf Introduction \label{Intro}}

Let $\mu$ be a finite positive Borel measure on a metric space with compact support. We call a collection $\mathcal{B}_{\delta}$ of disjoint closed $\delta$-balls with centers in ${\rm supp}(\mu)$ a \textit{$\delta$-packing} of ${\rm supp}(\mu)$. The \textit{lower $L^q$-spectrum}  of $\mu$ is defined for all $q\in\mathbb R$ as
\begin{equation}\label{eq:tau_def}
\underline{\tau}(q):=\varliminf_{\delta\rightarrow0^{+}}\frac{\log \mathcal{S}_{\delta}(q)}{\log\delta},
\end{equation}
where $\mathcal{S}_{\delta}(q):=\sup\sum_{B\in\mathcal{B}_{\delta}}\mu(B)^{q}$ and the supremum is taken over all $\delta$-packings of
${\rm supp}(\mu)$. The \textit{upper $L^q$-spectrum} of $\mu$ is defined analogously by replacing $\varliminf$ with $\varlimsup$. If the limit exists, we denote the common value by $\tau(q)$.
For $q\ne 1$, the \textit{lower and upper $L^q$-dimensions} of $\mu$ are defined as
\begin{equation}\label{eq:qdim_def}
\underline{\dim}_{q}(\mu):=\frac{\underline{\tau}(q)}{q-1}\quad\text{and}\quad\overline{\dim}_{q}(\mu):=\frac{\overline{\tau}(q)}{q-1},
\end{equation}
respectively (see, e.g., \cite{Strichartz_1993, Olsen_1995, Lau-Ngai_1999}).

The $L^q$-spectrum of a measure plays an important role in the multifractal formalism, which is a heuristic principle asserting that the Hausdorff dimension of the multifractal component consisting of points in ${\rm supp}(\mu)$ with local dimension being $\alpha$ is equal to the Legendre transform of $\tau(q)$, i.e.,
$$
\dim_{\rm H}\left(\left\{x\in {\rm supp}(\mu):\lim_{\delta\to 0^+}\frac{\log\mu(B_\delta(x))}{\log\delta}=\alpha\right\}\right)=\tau^*(\alpha),
$$
where we recall that $\tau^*(\alpha):=\inf\{\alpha q-\tau(q):q\in\mathbb R\}$.

For $\delta>0$, let $\mathcal{P}$ be a finite Borel $\delta$-partition of ${\rm supp}(\mu)$, i.e., the diameter of any element in $\mathcal{P}$ is at most $\delta$. Define
$$h(\mu,\mathcal{P}):=-\sum_{P\in\mathcal{P}}\mu(P)\log\mu(P).$$
Let
$$h(\mu,\delta):=\inf\{h(\mu,\mathcal{P}):\mathcal{P}~\text{is a finite Borel}~\delta\text{-partition of}~{\rm supp}(\mu)\}.$$
The \textit{entropy dimension} of $\mu$ is defined as
\begin{equation}\label{eq:edim_def}
\dim_{e}(\mu):=\lim_{\delta\rightarrow0^{+}}\frac{h(\mu,\delta)}{-\log\delta},
\end{equation}
if the limit exists.

Let $\mu$ be a finite positive Borel measure on a metric space with compact support $K$. The $L^q$-dimension (also called \textit{generalized dimension} in \cite{Patzschke_1997}) is defined by Hentschel and Procaccia \cite{Hentschel-Procaccia_1983} (see also \cite{Strichartz_1993}) as
\begin{equation}\label{eq:GD_def}
\lim_{\delta\rightarrow0^{+}}\frac{\log\int_{K}\mu(B_{\delta}(x))^{q-1}\,d\mu}{(q-1)\log\delta}, \quad q\ne 1,
\end{equation}
if the limit exists. For metric spaces with the doubling property, Guysinsky and Yaskolko \cite{Guysinsky-Yaskolko_1997} proved that for $q>1$, the definition in
(\ref{eq:GD_def}) is equivalent to the Renyi dimension, which is formulated by using the so-called grid partitions (see definition in Section 4).
For $0<q<1$ and measures on $\mathbb{R}^{n}$, Barbaroux \textit{et al.} \cite{Barbaroux-Germinet-Tcheremchantsev_2001} proved that the definitions of the $L^{q}$-dimension in (\ref{eq:qdim_def}) and (\ref{eq:GD_def}) are equivalent.
Moreover, for complete metric spaces and $q<0$, Germinet and Tcheremchantsev \cite{Germinet-Tcheremchantsev_2006} considered the equivalence of the definitions of the $L^{q}$-dimension in (\ref{eq:qdim_def}) and (\ref{eq:GD_def}).
\par We say that a metric space $X$ is \textit{doubling} if there exists $N_{0}\in\mathbb{N}$ such that any $2r$-ball can be covered by a union of at most $N_{0}$ balls of radius $r$. A measure $\mu$ is called \textit{doubling} on $X$ if there exists a constant $C\geq1$ called a \textit{doubling constant} such that for any $x\in X$ and $r>0$,
$$\mu(B_{2r}(x))\leq C\mu(B_{r}(x)).$$
In particular, for any $R\geq r>0$,
\begin{equation}\label{eq(1.1)}
\mu(B_{R}(x))\leq C^{\log_{2}(\frac{R}{r})}\mu(B_{r}(x)).
\end{equation}
Moreover, each metric space that carries a doubling measure must be doubling (see Proposition \ref{A_1}).

In this paper, the notions of heavy maximal packings and partitions (see definitions in Section \ref{sect.2}) will play an important role.
Roughly speaking, given a measure $\mu$ with compact support, we can use heavy maximal $\delta$-packings to define $\tau(q)$. Moreover, we can use a heavy maximal $\delta$-packing to construct a cover of the self-conformal set in such a way that each member of the cover corresponds to a member of the heavy maximal $\delta$-packing that has comparable $\mu$ measure.

We first show that for a doubling metric space, the expression in (\ref{eq:GD_def}) is equivalent to the $L^{q}$-dimension defined by (\ref{eq:qdim_def}) for any $q\in(0,\infty)\setminus\{1\}$. We obtain the equivalence for $q>1$ by showing that (\ref{eq:tau_def}) is equivalent to the Renyi dimension defined in \cite{Guysinsky-Yaskolko_1997}. The case $0<q<1$ is obtained by generalizing a result in \cite{Barbaroux-Germinet-Tcheremchantsev_2001}.

\begin{thm}\label{thm(1.0)}
Let $X$ be a doubling metric space. Assume that $\mu$ is a finite positive Borel measure on $X$ with compact support $K$. Then for any $q\in(0,\infty)\setminus\{1\}$, the definitions of the $L^{q}$-dimension in (\ref{eq:qdim_def}) and (\ref{eq:GD_def}) are equivalent.
\end{thm}

Assume that $\{S_{i}\}_{i=1}^{\ell}$ is a conformal iterated function system (CIFS) (see definition in Section \ref{sect.3}). Let $K=\bigcup_{i=1}^{\ell}S_{i}(K)$ be the self-conformal set, and $\mu=\sum_{i=1}^{\ell}p_{i}\mu\circ S_{i}^{-1}$ be the self-conformal measure with $K={\rm supp}(\mu)$, where $(p_{1},\dots,p_{\ell})$ is a probability vector.
Peres and Solomyak \cite{Peres-Solomyak_2000} proved that for any self-conformal measure on $\mathbb R^n$, the limits defining the $L^q$-dimension for any $q\in(0,\infty)\setminus\{1\}$, and the entropy dimension exist, without assuming any separation condition. The purpose of this paper is to obtain similar results for self-conformal measures on Riemannian manifolds.
Self-conformal measures on Riemannian manifolds have been studied by Patzschke \cite{Patzschke_1997}. By assuming the open set condition (OSC), Patzschke proved that the multifractal formalism holds for a self-conformal measure defined by a CIFS of $C^{1+\gamma}$ diffeomorphisms on a Riemannian manifold.

Our second objective in this paper is to generalize the result on the existence of the $L^q$-dimension in \cite{Peres-Solomyak_2000} to complete smooth Riemannian manifolds with the doubling property. We mention that we do not assume that the measure is doubling. In the proof of \cite{Peres-Solomyak_2000}, $\mathbb R^n$ is partitioned into $2^{-t}$-mesh cubes of the form
$$\mathcal P_t=\Bigg\{\prod_{i=1}^n[k_i 2^{-t},(k_i+1)2^{-t}),\quad k_i\in\mathbb Z \Bigg\},\quad t\in\mathbb N.$$
Any set of diameter less than $2^{-t}$ can intersect no more than $3^n$ members of $\mathcal P_t$. For Riemannian manifolds, such a partition is not possible. We use some properties equivalent to the doubling property (see Lemma \ref{lem(2.1)}) to deal with this problem.

\begin{thm}\label{thm(1.1)}
Let $M$ be a complete $n$-dimensional smooth Riemannian manifold with the doubling property. Assume that $\mu$ is a self-conformal measure on $M$ with compact support $K$. Then for $q>0$, the limit defining the $L^{q}$-spectrum in (\ref{eq:tau_def}) exists. In particular, for any $q\in(0,\infty)\setminus\{1\}$, the $L^q$-dimension defined by (\ref{eq:qdim_def}) exists.
\end{thm}

It follows from the Bishop-Gromov comparison theorem that a complete $n$-dimensional Riemannian manifold with nonnegative Ricci curvature has the doubling property (see, e.g., \cite{Berger_2003}). In fact, any $2r$-ball can be covered by no more than $4^n$ balls of radius $r$. Thus, the following corollary is a direct consequence of this fact, Theorem~\ref{thm(1.0)} and Theorem \ref{thm(1.1)}.

\begin{cor}
Let $M$ be a complete $n$-dimensional smooth Riemannian manifold with nonnegative Ricci curvature. Then the conclusions of Theorem \ref{thm(1.0)} and Theorem \ref{thm(1.1)} hold respectively.
\end{cor}

By assuming that $\mu$ is doubling, and using maximal partitions, we extend the result concerning the existence of the entropy dimension in \cite{Peres-Solomyak_2000} to complete smooth Riemannian manifolds.

\begin{thm}\label{thm(1.2)}
Let $M$ be a complete $n$-dimensional smooth Riemannian manifold. Assume that $\mu$ is a self-conformal doubling measure on $M$ with compact support $K$. Then the entropy dimension defined by (\ref{eq:edim_def}) exists.
\end{thm}

The rest of this paper is organized as follows. Section~\ref{sect.2} summarizes some basic definitions, facts, and preliminary results on metric spaces. We also prove several lemmas. Section~\ref{sect.3} summarizes the definition and properties of CIFSs on Riemannian manifolds. In Section~\ref{sect.4}, we prove that the definitions of the $L^{q}$-dimension of $\mu$ in (\ref{eq:qdim_def}) and (\ref{eq:GD_def}) are equivalent. Section~\ref{sect.5} and Section~\ref{sect.6} are devoted to the proof of Theorem~\ref{thm(1.1)} and Theorem~\ref{thm(1.2)} respectively. In Section~\ref{sect.7}, we present some examples of conformal measures on Riemannian manifolds satisfying the conditions of Theorem~\ref{thm(1.1)} and Theorem~\ref{thm(1.2)}. Finally, in the Appendix, we include for completeness the proof that each metric space that carries a doubling measure must be doubling, and the proof of Lemma \ref{lem(2.2)}.

\section{\bf Preliminaries \label{sect.2}}

Let $(X,\rho,\mu)$ be a metric measure space. Recall that a collection $\mathcal{B}_{\delta}$ of disjoint closed $\delta$-balls with centers in
$A\subset X$ is called a \textit{$\delta$-packing} of $A$. Denote the cardinality of $\mathcal{B}_{\delta}$ by $\#\mathcal{B}_{\delta}$. Let $t*B_{\delta}(x):=B_{t\delta}(x)$ for $t>0$ and $2\mathcal{B}_{\delta}:=\{2*B:B\in\mathcal{B}_{\delta}\}$. We say
that a $\delta$-packing $\mathcal{B}_{\delta}$ of $A$ is \textit{maximal} if for any $x\in A$ there exists $B\in\mathcal{B}_{\delta}$ intersects
$B_{\delta}(x)$. Obviously, if $\mathcal{B}_{\delta}$ is a maximal $\delta$-packing of $A$, then $2\mathcal{B}_{\delta}$ covers $A$.

Assume that $A$ is compact, and let $\{B_{\delta}(x_j)\}_{j=1}^{h}$ be a maximal $\delta$-packing of $A$. For any $j\in\{1,\dots,h\}$, let
\begin{equation}\label{eq:E_def}
F_{2\delta}^{j}:=B_{2\delta}(x_j)\setminus\bigcup_{i=1}^{j-1}B_{2\delta}(x_i)
\qquad\text{and}\qquad
E_{2\delta}^{j}:=B_{\delta}(x_{j})\cup
\left(F_{2\delta}^{j}\setminus\bigcup_{i=1}^{h}B_{\delta}(x_{i})\right)
\end{equation}
(see Figure \ref{fig.1}). Then
\begin{equation}\label{eq(2.02)}
F_{2\delta}^{j}\subset B_{2\delta}(x_{j})\qquad\text{and}\qquad B_{\delta}(x_{j})\subset E_{2\delta}^{j}\subset B_{2\delta}(x_{j}).
\end{equation}
Moreover,
\begin{equation}\label{eq(2.01)}
A\subset\bigcup_{j=1}^{h}E_{2\delta}^{j}\qquad\text{and}\qquad E_{2\delta}^{i}\cap E_{2\delta}^{j}=\emptyset,\quad\text{for}~i\neq j.
\end{equation}
For any $j\in\{1,\dots,h\}$, let
\begin{equation}\label{eq(2.04)}
\mathcal{E}_{2\delta}^{j}:=E_{2\delta}^{j}\cap A.
\end{equation}
Making use of (\ref{eq(2.02)}) and (\ref{eq(2.04)}), we have
\begin{equation}\label{eq(2.05)}
\mathcal{E}_{2\delta}^{j}\subset B_{2\delta}(x_{j}),
\end{equation}
and thus,
\begin{equation}\label{eq(2.06)}
{\rm diam}(\mathcal{E}_{2\delta}^{j})\leq 4\delta.
\end{equation}
It follows from (\ref{eq(2.01)}), (\ref{eq(2.04)}) and (\ref{eq(2.06)}) that $\mathbb{E}_{4\delta}:=\big\{\mathcal{E}_{2\delta}^{j}\big\}_{j=1}^{h}$ is a $4\delta$-partition of $A$, which is generated by $\{B_{\delta}(x_j)\}_{j=1}^{h}$. We
say that such a partition is \textit{maximal}.

The following obvious properties will be used repeatedly in our proofs. Let $X$ be a metric space, and $P\subset X$ with ${\rm diam}(P)\leq\eta$. It follows that $P\subset B_{\eta/2}(x)$ with some $x\in X$. Hence for any $y\in X$ and $r>0$ satisfying $P\cap B_{r}(y)\neq\emptyset$, we have
\begin{equation}\label{eq(2.801)}
x\in B_{\frac{\eta}{2}+r}(y)\quad\text{and}\quad y\in B_{\frac{\eta}{2}+r}(x).
\end{equation}
Moreover,
\begin{equation}\label{eq(2.802)}
P\subset B_{\eta+r}(y)\quad\text{and}\quad B_{r}(y)\subset B_{\frac{\eta}{2}+2r}(x).
\end{equation}
For any $j\in\{1,\dots,h\}$ satisfying $P\cap E_{2\delta}^{j}\neq\emptyset$ or $P\cap \mathcal{E}_{2\delta}^{j}\neq\emptyset$, it follows from (\ref{eq(2.02)}) and (\ref{eq(2.05)}) that
\begin{equation}\label{eq(2.803)}
x\in B_{\frac{\eta}{2}+2\delta}(x_{j})\quad\text{and}\quad x_{j}\in B_{\frac{\eta}{2}+2\delta}(x).
\end{equation}
Moreover,
\begin{equation}\label{eq(2.804)}
P\subset B_{\eta+2\delta}(x_{j})\quad\text{and}\quad E_{2\delta}^{j},\mathcal{E}_{2\delta}^{j}\subset B_{\frac{\eta}{2}+4\delta}(x).
\end{equation}

\begin{figure}[H]
\centering
\begin{subfigure}{0.49\linewidth}
  \centering
  \includegraphics[width=6cm]{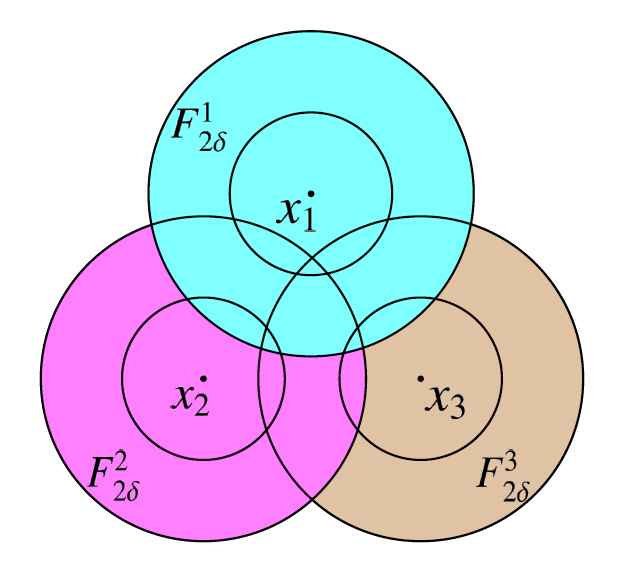}\\
  \caption*{(a)}\label{fig.(a)}
  \end{subfigure}
  \centering
  \begin{subfigure}{0.49\linewidth}
  \centering
  \includegraphics[width=6cm]{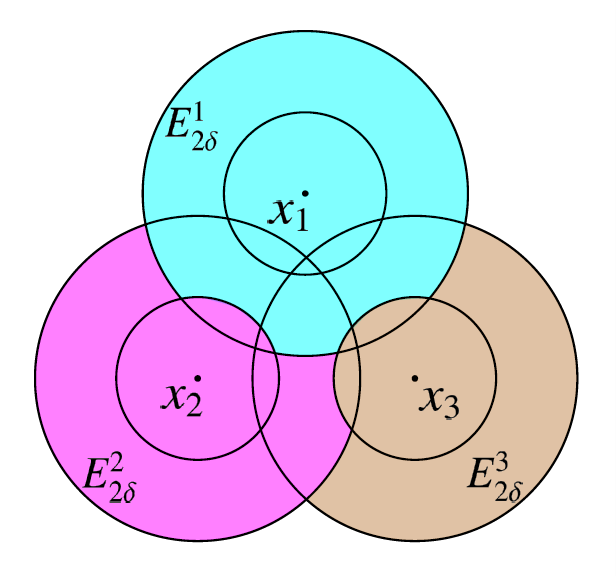}\\
  \caption*{(b)}\label{fig.(b)}
  \end{subfigure}
  \caption{Illustration of the sets in (\ref{eq:E_def}). (a) The sets $F_{2\delta}^{1}$, $F_{2\delta}^{2}$ and $F_{2\delta}^{3}$.
  (b) The sets $E_{2\delta}^{1}$, $E_{2\delta}^{2}$ and $E_{2\delta}^{3}$.}\label{fig.1}
\end{figure}

\begin{defi}\label{defi(2.1)}
Let $(X,\rho,\mu)$ be a metric measure space and $A\subset X$ be compact. We say that a maximal $\delta$-packing $\{B_{\delta}(x_j)\}_{j=1}^{h}$ of $A$ is
\textbf{heavy} if $\mu(B_{\delta}(x))\leq2\mu(B_{\delta}(x_1))$ for any $x\in A$, and $\mu(B_{\delta}(x))\leq2\mu(B_{\delta}(x_j))$ for any
$x\in A\setminus\bigcup_{i=1}^{j-1}B_{2\delta}(x_{i})$ and $j\in\{2,\dots,h\}$.
In addition, a maximal $4\delta$-partition $\mathbb{E}_{4\delta}$ of $A$ is called \textbf{heavy} if it is generated by a heavy maximal $\delta$-packing of $A$.
\end{defi}

The following lemma is well known (see, e.g., \cite{Heinonen_2001, Antti_2014}).

\begin{lem}\label{lem(2.1)}
For a metric space $X$, the following statements are equivalent.\\
$(a)$ $X$ is doubling;\\
$(b)$ there are two constants $p>0$ and $D_{0}\geq1$ such that for all $R\geq r>0$, any ball of radius $R$ can be covered by $D_{0}(R/r)^{p}$ balls of radius $r$;\\
$(c)$ there are two constants $p>0$ and $D_{0}\geq1$ such that if $R\geq r>0$, then the cardinality of any $r$-packing of a ball of radius $R$ is at most $D_{0}(R/r)^{p}$.
\end{lem}

By Lemma \ref{lem(2.1)}, we obtain the following proposition.

\begin{prop}\label{prop(2.1)}
Let $(X,\rho,\mu)$ be a metric measure space with the doubling property and $A\subset X$ be compact. Assume that $\{B_{\delta}(x_j)\}_{j=1}^{h}$ is a heavy maximal $\delta$-packing of $A$. Let $\{F_{2\delta}^{j}\}_{j=1}^{h}$ and $\{E_{2\delta}^{j}\}_{j=1}^{h}$ be defined with respect to $\{B_{\delta}(x_j)\}_{j=1}^{h}$ as in (\ref{eq:E_def}). Then there exists a constant $C_{1}\geq1$ such that for any $\delta>0$ and $j\in\{1,\dots,h\}$, $\mu(E_{2\delta}^{j})\leq C_{1}\mu(B_{\delta}(x_{j}))$.
\end{prop}
\begin{proof}
For $j=1$, (\ref{eq(2.02)}) and Lemma \ref{lem(2.1)}(b) imply that there exist $D_{0}2^{p}$ $\delta$-balls such that $E_{2\delta}^{1}\subset B_{2\delta}(x_{1})\subset\bigcup_{i=1}^{D_{0}2^{p}}B_{\delta}(x_{1_{i}})$. It follows from the definition of a heavy maximal packing that for any $i\in\{1,\dots,D_{0}2^{p}\}$,
$$\mu(B_{\delta}(x_{1_{i}}))\leq2\mu(B_{\delta}(x_{1})).$$
Hence
$$\mu(E_{2\delta}^{1})\leq\sum_{i=1}^{D_{0}2^{p}}\mu(B_{\delta}(x_{1_{i}}))\leq D_{0}2^{p+1}\mu(B_{\delta}(x_{1})).$$
For any $j\in\{2,\dots,h\}$, let $\mathcal{B}$ be a maximal $(\delta/2)$-packing of $F_{2\delta}^{j}\setminus\bigcup_{i=1}^{h}B_{\delta}(x_{i})$. Then $2\mathcal{B}$ covers $F_{2\delta}^{j}\setminus\bigcup_{i=1}^{h}B_{\delta}(x_{i})$. Making use of (\ref{eq(2.02)}), we have
$F_{2\delta}^{j}\setminus\bigcup_{i=1}^{h}B_{\delta}(x_{i})\subset B_{2\delta}(x_{j})$.
By Lemma \ref{lem(2.1)}(c), the cardinality of any $(\delta/2)$-packing of $B_{2\delta}(x_{j})$ is at most $D_{0}4^{p}$.
Hence $\#\mathcal{B}\leq D_{0}4^{p}$. Since the center of any ball $B\in2\mathcal{B}$  lies in $A\setminus\bigcup_{i=1}^{j-1}B_{2\delta}(x_{i})$, it follows from the definition of a heavy maximal packing that
$$\mu(B)\leq2\mu(B_{\delta}(x_{j})).$$
Hence
$$\mu(E_{2\delta}^{j})\leq\mu(B_{\delta}(x_{j}))+\sum_{B\in2\mathcal{B}}\mu(B)\leq (1+D_{0}2^{2p+1})\mu(B_{\delta}(x_{j})).$$
The proposition follows by letting $C_{1}:=1+D_{0}2^{2p+1}\geq1$.
\end{proof}

We need the following lemma. For completeness, we include a proof of Lemma \ref{lem(2.2)} in the Appendix.

\begin{lem}\label{lem(2.2)}
Let $q>0$ and $k\in\mathbb{N}$. Then there exists a constant $A_{0}>0$ such that for any sequence $\{a_{i}\}_{i=1}^{k}$ of nonnegative numbers, we have
$\big(\sum_{i=1}^{k}a_{i}\big)^{q}\leq A_{0}\sum_{i=1}^{k}a_{i}^{q}$. In particular, we may take $A_{0}:=\max\{k^{q-1},1\}$.
\end{lem}

Let $X$ be a metric space, and $\mu$ be a finite positive Borel measure on $X$ with compact support $K$. For $q\in\mathbb{R}$, define
\begin{equation}\label{S_def}
\mathcal{S}_{2^{-t}}(q):=\sup\Bigg\{\sum_{B\in\mathcal{B}_{2^{-t}}}\mu(B)^{q}: \mathcal{B}_{2^{-t}}~\text{is a}~2^{-t}\text{-packing of}~K\Bigg\}.
\end{equation}
The lower $L^{q}$-spectrum of $\mu$ can be defined equivalently as
\begin{equation}\label{eq(2.8)}
\underline{\tau}(q):=\varliminf_{t\rightarrow\infty}\frac{\log \mathcal{S}_{2^{-t}}(q)}{-t\log2},
\end{equation}
the same holds for $\overline{\tau}(q)$ and thus $\tau(q)$.
Let
\begin{equation}\label{S^*_def}
\mathcal{S}^{*}_{2^{-t}}(q):=
\sup\Bigg\{\sum_{B\in\mathcal{B}_{2^{-t}}^{*}}\mu(B)^{q}: \mathcal{B}_{2^{-t}}^{*}~\text{is a heavy maximal}~2^{-t}\text{-packing of}~K\Bigg\}.
\end{equation}

The following proposition shows that $\mathcal{S}_{2^{-t}}(q)$ and $\mathcal{S}^{*}_{2^{-t}}(q)$ control each other.

\begin{prop}\label{prop(2.2)}
Let $X$ be a doubling metric space. Assume that $\mu$ is a finite positive Borel measure on $X$ with compact support $K$. Let $\mathcal{S}_{2^{-t}}(q)$ and $\mathcal{S}^{*}_{2^{-t}}(q)$ be defined as in (\ref{S_def}) and (\ref{S^*_def}) respectively. Then there exists a constant $C_{2}\geq1$ such that for any $q>0$,
$$
\mathcal{S}^{*}_{2^{-t}}(q)\leq\mathcal{S}_{2^{-t}}(q)\leq C_{2}\mathcal{S}^{*}_{2^{-t}}(q).
$$
\end{prop}

\begin{proof}
It follows from definition that
\begin{equation}\label{eq(2.4)}
\mathcal{S}^{*}_{2^{-t}}(q)\leq\mathcal{S}_{2^{-t}}(q).
\end{equation}
To prove the reverse inequality, let $\mathcal{B}^{*}_{2^{-t}}:=\{B_{2^{-t}}(x_{j})\}_{j=1}^{h}$ be a heavy maximal $2^{-t}$-packing of $K$, and let $\{E_{2^{-t+1}}^{j}\}_{j=1}^{h}$ be defined with respect to $\mathcal{B}^{*}_{2^{-t}}$ as in (\ref{eq:E_def}). Making use of (\ref{eq(2.02)}), we have $E_{2^{-t+1}}^{j}\subset B_{2^{-t+1}}(x_{j})$ for $j\in\{1,\dots,h\}$. Let $\mathcal{B}_{2^{-t}}$ be a $2^{-t}$-packing of $K$.
For $B\in\mathcal{B}_{2^{-t}}$, let $x$ be the center of $B$. If $B\cap E_{2^{-t+1}}^{j}\neq\emptyset$, then by (\ref{eq(2.803)}), we have
\begin{equation}\label{eq(2.41)}
x_{j}\in 3*B\quad\text{and}\quad x\in B_{3\cdot2^{-t}}(x_{j}).
\end{equation}
By Lemma \ref{lem(2.1)}(c), the cardinality of any $2^{-t}$-packing of $3*B$ is at most $D_{0}3^{p}$. Hence $B$ intersects at most $D_{0}3^{p}$ elements of $\{E_{2^{-t+1}}^{j}\}_{j=1}^{h}$.
Let $A_{1}:=\max\{D_{0}^{q-1}3^{p(q-1)},1\}$. Since $B\subset\bigcup_{j:B\cap E_{2^{-t+1}}^{j}\neq\emptyset}E_{2^{-t+1}}^{j}$, it follows from Lemma \ref{lem(2.2)} that for $q>0$,
\begin{equation}\label{eq(2.5)}
\mu(B)^{q}\leq\Bigg(\sum_{j:B\cap E_{2^{-t+1}}^{j}\neq\emptyset}\mu(E_{2^{-t+1}}^{j})\Bigg)^{q}
\leq A_{1}\sum_{j:B\cap E_{2^{-t+1}}^{j}\neq\emptyset}\mu(E_{2^{-t+1}}^{j})^{q}.
\end{equation}
By Lemma \ref{lem(2.1)}(c), the cardinality of any $2^{-t}$-packing of $B_{3\cdot2^{-t}}(x_{j})$ is at most $D_{0}3^{p}$. In view of this and (\ref{eq(2.41)}), we see that $E_{2^{-t+1}}^{j}$ intersects at most $D_{0}3^{p}$ elements of $\mathcal{B}_{2^{-t}}$. Thus, for $q>0$,
\begin{equation}\label{eq(2.6)}\begin{aligned}
\sum_{B\in\mathcal{B}_{2^{-t}}}\mu(B)^{q}
&\leq A_{1}\sum_{B\in\mathcal{B}_{2^{-t}}}\sum_{j:B\cap E_{2^{-t+1}}^{j}\neq\emptyset}\mu(E_{2^{-t+1}}^{j})^{q}\quad(\text{by}~(\ref{eq(2.5)}))\\
&\leq D_{0}A_{1}3^{p}\sum_{j=1}^{h}\mu(E_{2^{-t+1}}^{j})^{q}\\
&\leq D_{0}A_{1}C_{1}^{q}3^{p}\sum_{j=1}^{h}\mu(B_{2^{-t}}(x_{j}))^{q}\quad(\text{by Proposition}~\ref{prop(2.1)}).
\end{aligned}\end{equation}
Let $C_{2}:=D_{0}A_{1}C_{1}^{q}3^{p}\geq1$. It follows from (\ref{eq(2.6)}) that
\begin{equation}\label{eq(2.7)}
\mathcal{S}_{2^{-t}}(q)\leq C_{2}\mathcal{S}^{*}_{2^{-t}}(q).
\end{equation}
Combining (\ref{eq(2.4)}) and (\ref{eq(2.7)}) yields the proposition.
\end{proof}

We adopt the following definition from \cite{Peres-Solomyak_2000}.

\begin{defi}\label{defi(2.1)}
Let $X$ be a metric space and $A\subset X$ be compact. For fix $Q,\delta>0$, $D\in\mathbb{N}$, we say that a cover $\{G_{i}\}_{i=1}^{k}$ of $A$ by Borel sets is $\boldsymbol{(Q,\delta,D)}$-\textbf{good} if ${\rm diam}(G_{i})\leq Q\delta$, $G_{i}\cap A\neq\emptyset$ for any $i\in\{1,\dots,k\}$, and any $\delta$-ball intersects at most $D$ elements of the cover.
\end{defi}

Let $\mu$ be a finite positive Borel measure with compact support $K$. For a good cover of $K$, we have the following lemma.

\begin{lem}\label{lem(2.3)}
Assume the same hypotheses of Proposition \ref{prop(2.2)}. Let $Q,q>0$ and $D\in\mathbb{N}$. Then there exists a constant $C_{3}\geq1$ such that for any $t\in\mathbb{N}$ and $(Q,2^{-t},D)$-good cover $\{G_{i}\}_{i=1}^{k}$ of $K$,
\begin{equation}\label{eq(2.101)}
C_{3}^{-1}\mathcal{S}_{2^{-t}}(q)\leq\sum_{i=1}^{k}\mu(G_{i})^{q}\leq C_{3}\mathcal{S}_{2^{-t}}(q).
\end{equation}
\end{lem}

\begin{proof}
Let $\{B_{2^{-t}}(x_{j})\}_{j=1}^{h}$ be a heavy maximal $2^{-t}$-packing of $K$ with $t\in\mathbb{N}$, and let $\{E_{2^{-t+1}}^{j}\}_{j=1}^{h}$ be defined with respect to $\{B_{2^{-t}}(x_{j})\}_{j=1}^{h}$ as in (\ref{eq:E_def}). Let $\tilde{Q}:=\max\{1,Q\}$ and $m\in\mathbb{N}$ such that $\tilde{Q}\leq2^{m}$. Then ${\rm diam}(G_{i})\leq Q\cdot2^{-t}\leq\tilde{Q}\cdot2^{-t}\leq2^{m-t}$, where $i\in\{1,\dots,k\}$. Hence $G_{i}\subset B_{2^{m-t-1}}(x)\subset B_{2^{m-1}}(x)$ for some $x\in X$. For any $j\in\{1,\dots,h\}$, if $E_{2^{-t+1}}^{j}\cap G_{i}\neq\emptyset$, then by (\ref{eq(2.803)}), $x_{j}\in B_{(2^{m-1}+2)\cdot2^{-t}}(x)$. By Lemma \ref{lem(2.1)}(c), the cardinality of any $2^{-t}$-packing of $B_{(2^{m-1}+2)\cdot2^{-t}}(x)$ is at most $D_{0}(2^{m-1}+2)^{p}$. Hence $G_{i}$ intersects at most $D_{0}(2^{m-1}+2)^{p}$ elements of $\{E_{2^{-t+1}}^{j}\}_{j=1}^{h}$. Let $A_{2}:=\max\{D_{0}^{q-1}(2^{m-1}+2)^{p(q-1)},1\}$. Since $G_{i}\subset\bigcup_{j:G_{i}\cap E_{2^{-t+1}}^{j}\neq\emptyset}E_{2^{-t+1}}^{j}$, it follows from Lemma \ref{lem(2.2)} that for $q>0$,
\begin{equation}\label{eq(2.10)}
\mu(G_i)^{q}\leq\Bigg(\sum_{j:G_{i}\cap E_{2^{-t+1}}^{j}\neq\emptyset}\mu(E_{2^{-t+1}}^{j})\Bigg)^{q}
\leq A_{2}\sum_{j:G_{i}\cap E_{2^{-t+1}}^{j}\neq\emptyset}\mu(E_{2^{-t+1}}^{j})^{q}.
\end{equation}
Making use of (\ref{eq(2.02)}), we have $E_{2^{-t+1}}^{j}\subset B_{2^{-t+1}}(x_{j})$. By Lemma \ref{lem(2.1)}(b), $B_{2^{-t+1}}(x_{j})$ can be covered by a union of at most $D_{0}2^{p}$ balls of radius $2^{-t}$. It follows from the definition of a good cover that any $2^{-t}$-ball intersects at most $D$ elements of $\{G_{i}\}_{i=1}^{k}$. Hence $E_{2^{-t+1}}^{j}$ intersects at most $DD_{0}2^{p}$ elements of $\{G_{i}\}_{i=1}^{k}$. Thus,
\begin{equation}\label{eq(2.11)}\begin{aligned}
\sum_{i=1}^{k}\mu(G_{i})^q
&\leq A_{2}\sum_{i=1}^{k}\sum_{j:G_{i}\cap E_{2^{-t+1}}^{j}\neq\emptyset}\mu(E_{2^{-t+1}}^{j})^{q}\quad(\text{by}~(\ref{eq(2.10)}))\\
&\leq DD_{0}A_{2}2^{p}\sum_{j=1}^{h}\mu(E_{2^{-t+1}}^{j})^{q}\\
&\leq DD_{0}A_{2}C_{1}^{q}2^{p}\sum_{j=1}^{h}\mu(B_{2^{-t}}(x_{j}))^q\quad(\text{by Proposition}~\ref{prop(2.1)}).
\end{aligned}\end{equation}
Taking supremum on the right side of (\ref{eq(2.11)}), we obtain from  Proposition \ref{prop(2.2)} that
\begin{equation}\label{eq(2.12)}
\sum_{i=1}^{k}\mu(G_{i})^q\leq DD_{0}A_{2}C_{1}^{q}2^{p}\mathcal{S}_{2^{-t}}(q):=c_{1}\mathcal{S}_{2^{-t}}(q).
\end{equation}

On the other hand, it follows from the definition of a good cover that $B_{2^{-t}}(x_{j})$ intersects at most $D$ elements of $\{G_{i}\}_{i=1}^{k}$. Let $A_{3}:=\max\{D^{q-1},1\}$. Since $B_{2^{-t}}(x_{j})\cap K\subset\bigcup_{i:B_{2^{-t}}(x_{j})\cap G_{i}\neq\emptyset}G_{i}$, we obtain from Lemma \ref{lem(2.2)} that for $q>0$,
\begin{equation}\label{eq(2.13)}
\mu(B_{2^{-t}}(x_j))^q\leq\Bigg(\sum_{i:B_{2^{-t}}(x_{j})\cap G_{i}\neq\emptyset}\mu(G_{i})\Bigg)^q
\leq A_{3}\sum_{i:B_{2^{-t}}(x_{j})\cap G_{i}\neq\emptyset}\mu(G_{i})^q.
\end{equation}
Note that $G_{i}\subset B_{2^{m-t-1}}(x)$ for some $x\in X$. Hence if $B_{2^{-t}}(x_j)\cap G_{i}\neq\emptyset$, then by (\ref{eq(2.801)}), $x_{j}\in B_{(2^{m-1}+1)\cdot2^{-t}}(x)$. By Lemma \ref{lem(2.1)}(c), the cardinality of any $2^{-t}$-packing of $B_{(2^{m-1}+1)\cdot2^{-t}}(x)$ is at most $D_{0}(2^{m-1}+1)^{p}$. Hence $G_{i}$ intersects at most $D_{0}(2^{m-1}+1)^{p}$ elements of $\{B_{2^{-t}}(x_{j})\}_{j=1}^{h}$. It follows from (\ref{eq(2.13)}) that
\begin{equation}\label{eq(2.14)}
\sum_{j=1}^{h}\mu(B_{2^{-t}}(x_j))^q\leq A_{3}\sum_{j=1}^{h}\sum_{i:B_{2^{-t}}(x_{j})\cap G_{i}\neq\emptyset}\mu(G_{i})^q
\leq D_{0}A_{3}(2^{m-1}+1)^{p}\sum_{i=1}^{k}\mu(G_{i})^q.
\end{equation}
Taking supremum on the left side of (\ref{eq(2.14)}), we obtain from  Proposition \ref{prop(2.2)} that
\begin{equation}\label{eq(2.15)}
\mathcal{S}_{2^{-t}}(q)\leq D_{0}C_{2}A_{3}(2^{m-1}+1)^{p}\sum_{i=1}^{k}\mu(G_{i})^q:=c_{2}\sum_{i=1}^{k}\mu(G_{i})^q.
\end{equation}
Now the lemma follows by combining (\ref{eq(2.12)}) and (\ref{eq(2.15)}) with $C_{3}:=c_{1}c_{2}\geq1$.
\end{proof}

\begin{rem}\label{rem(2.1)}
Replacing $2^{-t}$, $t\in\mathbb{N}$, by $\delta>0$, we obtain the analog of Lemma \ref{lem(2.3)} with (\ref{eq(2.101)}) taking the form
$$C_{3}^{-1}\mathcal{S}_{\delta}(q)\leq\sum_{i=1}^{k}\mu(G_{i})^{q}\leq C_{3}\mathcal{S}_{\delta}(q).$$
\end{rem}

\section{\bf Conformal iterated function systems on Riemannian manifolds \label{sect.3}}

Let $M$ be a complete $n$-dimensional smooth Riemannian manifold. Denote the Riemannian distance in $M$ by $d(\cdot,\cdot)$. Let $\{(\varphi_{\alpha},U_{\alpha})\}_{\alpha\in \mathcal{A}}$ be an atlas of $M$. We say that $S:M\rightarrow M$ is a \textit{$C^{1+\gamma}$ diffeomorphism} at $x\in M$ if there exist $\alpha,\beta\in \mathcal{A}$ such that $x\in U_{\alpha}$, $S(x)\in U_{\beta}$ and
$$f:=\varphi_{\beta}\circ S\circ\varphi_{\alpha}^{-1}:\mathbb{R}^n\rightarrow\mathbb{R}^n$$
is $C^{1+\gamma}$. We say $S$ is $C^{1+\gamma}$ on $M$ if it is a $C^{1+\gamma}$ diffeomorphism at each $x\in M$. The definition of
$C^{1+\gamma}$ diffeomorphism is independent of the choice of the coordinate charts. Denote the differential of $S$ by $S'$. Then $S'=f'$.

Assume that $U\subset M$ is open and connected, and $W\subset U$ is a compact set with $\overline{{\rm int}~W}=W$. Recall that a map $S:U\rightarrow U$ is called \textit{conformal} if $S'(x)$ is a similarity matrix for any $x\in U$.
Let $\Sigma:=\{1,\dots,\ell\}$, where $\ell\in\mathbb{N}$ and $\ell\geq2$. We say that $\{S_{i}\}_{i=1}^{\ell}$ is a \textit{conformal iterated function system (CIFS)}
on $U$, if\\
(a) $S_{i}:U\rightarrow S_{i}(U)\subset U$ is a conformal $C^{1+\gamma}$ diffeomorphism with $0<\gamma<1$ for any $i\in\Sigma$;\\
(b) $S_{i}(W)\subset W$ for any $i\in\Sigma$;\\
(c) $0<|\det S'_{i}(x)|<1$ for any $i\in\Sigma$ and $x\in U$.

We note here that unlike \cite{Patzschke_1997}, we do not assume that $\{S_{i}\}_{i=1}^{\ell}$ satisfies OSC. Then by Hutchinson \cite{Hutchinson_1981}, there exists a unique nonempty compact set $K\subset W$ called the \textit{self-conformal set} satisfying $K=\bigcup_{i=1}^{\ell}S_{i}(K)$. Moreover, given a probability vector $(p_{1},\dots,p_{\ell})$, there exists a unique Borel probability measure $\mu$ called the \textit{self-conformal measure} such that $\mu=\sum_{i=1}^{\ell}p_{i}\mu\circ S_{i}^{-1}$ and $K={\rm supp}(\mu)$. Let $\Sigma^{\ast}:=\bigcup_{k\geq1}\Sigma^{k}$, where $k\in\mathbb{N}$. For $u=(u_1,\dots,u_k)\in\Sigma^{k}$, let $u^{-}:=(u_{1},\dots,u_{k-1})$ and write $S_u:=S_{u_1}\circ\cdots\circ S_{u_k}$ and $p_u:=p_{u_1}\cdots p_{u_k}$. Define
$$\mathcal{W}_{t}:=\{u\in\Sigma^{\ast}:{\rm diam}(K_{u})\leq2^{-t},~{\rm diam}(K_{u^{-}})>2^{-t}\}.$$
Then for any $t\in\mathbb{N}$,
\begin{equation}\label{eq(2.16)}
\mu=\sum_{u\in \mathcal{W}_{t}}p_{u}\mu\circ S_{u}^{-1}.
\end{equation}

Since $M$ is a manifold, we can find an open and connected set $V$ such that $\overline{V}$ is compact and $W\subset V\subset \overline{V}\subset U$. According to \cite{Patzschke_1997}, the conditions of a CIFS imply the \textit{bounded distortion property}, without assuming any separation condition, i.e., there exists a constant $D_{1}\geq1$ such that for any $u\in\Sigma^{*}$ and $x,y\in V$,
\begin{equation}\label{eq(2.1)}
D_{1}^{-1}\leq\frac{|\det S'_{u}(x)|}{|\det S'_{u}(y)|}\leq D_{1}.
\end{equation}
Let $\|S'_{u}\|:=\sup_{x\in W}|\det S'_{u}(x)|$. Then there exists a constant $D_{2}\geq D_{1}$ such that for any $u\in\Sigma^{\ast}$ and $x,y\in W$,
\begin{equation}\label{eq(2.2)}
D_{2}^{-1}\|S'_{u}\|d(x,y)\leq d(S_{u}(x),S_{u}(y))\leq D_{2}\|S'_{u}\|d(x,y).
 \end{equation}
Let $K_{u}:=S_{u}(K)$. By (\ref{eq(2.2)}), we have
\begin{equation}\label{eq(2.3)}
D_{3}^{-1}\|S'_{u}\|\leq {\rm diam}(K_{u})\leq D_{3}\|S'_{u}\|,
\end{equation}
where $D_{3}\geq1$ is any finite number $\geq\max\{D_{2}{\rm diam}(K),D_{2}/{\rm diam}(K)\}$.

We can use the following lemma to construct a good cover of $K$ from any maximal packing of $K$.
\begin{lem}\label{lem(2.4)}
Let $M$ be a complete $n$-dimensional smooth Riemannian manifold with the doubling property. Let $\{B_{2^{-s-t}}(x_{j})\}_{j=1}^{h}$ be a maximal $2^{-s-t}$-packing of $K$ with $s,t\in\mathbb{N}$, and let $\{E_{2^{-s-t+1}}^{j}\}_{j=1}^{h}$ be defined with respect to $\{B_{2^{-s-t}}(x_{j})\}_{j=1}^{h}$ as in (\ref{eq:E_def}). Then there exists $Q>0$ and $D,N\in\mathbb{N}$ such that for any $t\in\mathbb{N}$, $s\geq N$ and $u\in \mathcal{W}_{t}$, the collection
$\mathcal{C}:=\{S_{u}^{-1}(E_{2^{-s-t+1}}^{j})\cap K\}_{j=1}^{h}$ is a $(Q,2^{-s},D)$-good cover of $K$.
\end{lem}
\begin{proof}
In view of (\ref{eq(2.01)}), $\mathcal{C}$ is obviously a cover of $K$. It follows from (\ref{eq(2.1)}) and (\ref{eq(2.3)}) that
\begin{equation}\label{eq(2.17)}
2^{-t}<{\rm diam}(K_{u^{-}})\leq D_{3}\|S_{u^{-}}^{\prime}\|\leq D_{1}D_{3} \lambda^{-1}\|S_{u}^{\prime}\|,
\end{equation}
where $\lambda:=\min\{|\det S_{i}^{\prime}(x)|:i\leq\ell~\text{and}~x\in W\}>0$.
\draft{For any $u=(u_{1},\dots,u_{k})$,
$$\begin{aligned}
D_{1}|\det S'_{u}(y)|\geq|\det S'_{u}(x)|&=|\det S'_{u^{-}\circ u_{k}}(x)|
=|\det S'_{u^{-}}(S_{u_{k}}(x))\cdot S'_{u_{k}}(x)|\\
&=|\det S'_{u^{-}}(S_{u_{k}}(x))|\cdot |\det S'_{u_{k}}(x)|\\
&\geq \lambda|\det S'_{u^{-}}(S_{u_{k}}(x))|.
\end{aligned}$$
$$\Longrightarrow D_{1}\parallel S'_{u}\parallel\geq \lambda\parallel S'_{u^{-}}\parallel$$
$$\Longrightarrow D_{1}\lambda^{-1}\parallel S'_{u}\parallel\geq\parallel S'_{u^{-}}\parallel$$}
Hence we obtain from (\ref{eq(2.2)}) and (\ref{eq(2.17)}) that for
any $j\in\{1,\dots,h\}$,
$${\rm diam}(S_{u}^{-1}(E_{2^{-s-t+1}}^{j})\cap K)\leq D_{2}\|S'_{u}\|^{-1}{\rm diam}(E_{2^{-s-t+1}}^{j})\leq D_{1}D_{2}D_{3}\lambda^{-1}2^{t}\cdot2^{-s-t+2}
:=Q\cdot2^{-s}.$$
This proves the first property of a good cover.

Let $B$ be any  $2^{-s}$-ball satisfying $B\cap K\neq\emptyset$; otherwise the second property of a good cover holds trivially.
Without loss of generality, we can assume that $s\geq N$ is large enough such that $B\subset W$.
It follows from (\ref{eq(2.2)}) and (\ref{eq(2.3)}) that
$${\rm diam}(S_{u}(B))\leq D_{2}\|S'_{u}\|{\rm diam}(B)
\leq D_{2}D_{3}{\rm diam}(K_{u}){\rm diam}(B)
\leq D_{2}D_{3}2^{-s-t+1}.$$
Consequently, $S_{u}(B)\subset B_{D_{2}D_{3}2^{-s-t}}(x)$ for some $x\in M$. For any $j\in\{1,\dots,h\}$,
if $E_{2^{-s-t+1}}^{j}\cap S_{u}(B)\neq\emptyset$, then by (\ref{eq(2.803)}), $x_{j}\in B_{(D_{2}D_{3}+2)2^{-s-t}}(x)$.
By Lemma \ref{lem(2.1)}(c), the cardinality of
any $2^{-s-t}$-packing of $B_{(D_{2}D_{3}+2)2^{-s-t}}(x)$ is at most $D_{0}(D_{2}D_{3}+2)^{p}$. Hence $S_{u}(B)$ intersects at most $D_{0}(D_{2}D_{3}+2)^{p}$ elements of $\{E_{2^{-s-t+1}}^{j}\}_{j=1}^{h}$. Therefore, $B$ intersects at most $D_{0}(D_{2}D_{3}+2)^{p}$ elements of $\mathcal{C}$. This completes the proof.
\end{proof}

\section{\bf An equivalent definition of \boldmath$L^q$-dimension \label{sect.4}}

In this section, we prove that for any $q\in(0,\infty)\setminus\{1\}$, the definitions of the $L^{q}$-dimension in (\ref{eq:tau_def}) and (\ref{eq:GD_def}) are equivalent. We first give a definition.

\begin{defi}\label{defi(4.1)}(\cite{Guysinsky-Yaskolko_1997})
Let $0<\lambda<1$, $\delta>0$, and $X$ be a metric space. A partition $\{P_{i}\}_{i=1}^{\infty}$ of $X$ is called a $\boldsymbol{(\lambda,\delta)}$-\textbf{grid} if there exists a sequence $\{x_{i}\}_{i=1}^{\infty}$ in $X$ such that for any $i\in\mathbb{N}$,
$$B_{\lambda\delta}(x_{i})\subset P_{i}\subset B_{\delta}(x_{i}).$$
\end{defi}

Let $\mu$ be a finite positive Borel measure on $X$. Fix $0<\lambda<1$ and for any $\delta>0$, let $\{P_{i}\}_{i=1}^{\infty}$ be a $(\lambda,\delta)$-grid partition of $X$. The \textit{Renyi dimension} in \cite{Guysinsky-Yaskolko_1997} is defined as
$$\lim_{\delta\rightarrow0^{+}}\frac{\log\sup\sum_{i=1}^{\infty}\mu(P_{i})^{q}}{(q-1)\log\delta},\quad q\neq1,$$
if the limit exists, where the supremum is taken over all $(\lambda,\delta)$-grid partitions of $X$. We have the following proposition.

\begin{prop}\label{prop(4.1)}
Let $X$ and $\mu$  be as in Theorem \ref{thm(1.0)}. Then for any $q\in(0,\infty)\setminus\{1\}$, the Renyi dimension is equivalent to the $L^{q}$-dimension defined by (\ref{eq:tau_def}).
\end{prop}

\begin{proof}
Let $\{P_{i}\}_{i=1}^{\infty}$ be a $(\lambda,\delta)$-grid partition of $X$. Then $K\subset\bigcup_{i:P_{i}\cap K\neq\emptyset}P_{i}$. It follows from definition that $P_{i}\subset B_{\delta}(x_{i})$ and ${\rm diam}(P_{i})\leq2\delta$ for any $i\in\mathbb{N}$.  Hence for any $B_{\delta}(x)\subset X$, if $P_{i}\cap B_{\delta}(x)\neq\emptyset$, then by (\ref{eq(2.801)}), $x_{i}\in B_{2\delta}(x)$. Note that any $P_{i}$ contains
a $\lambda\delta$-ball. By Lemma \ref{lem(2.1)}(c), the cardinality of any $\lambda\delta$-packing of $B_{2\delta}(x)$ is at most $D_{0}2^{p}\lambda^{-p}$. Hence $B_{\delta}(x)$ intersects at most $D_{0}2^{p}\lambda^{-p}$ elements in $\{P_{i}\}_{i=1}^{\infty}$. It follows from the definition of a good cover that $\{P_{i}:P_{i}\cap K\neq\emptyset,~i\in\mathbb{N}\}$ is a $(2,\delta,D_{0}2^{p}\lambda^{-p})$-good cover of $K$. By Remark \ref{rem(2.1)}, there exists a constant $C_{3}\geq1$ such that for any $q>0$,
$$C_{3}^{-1}\mathcal{S}_{\delta}(q)\leq\sum_{i:P_{i}\cap K\neq\emptyset}\mu(P_{i})^{q}\leq C_{3}\mathcal{S}_{\delta}(q).$$
The proposition follows.
\end{proof}

\begin{proof}[Proof of Theorem \ref{thm(1.0)}]
The case $q>1$ follows from Proposition \ref{prop(4.1)} and \cite{Guysinsky-Yaskolko_1997} (see Theorem 4.1). We  only need to prove the case $0<q<1$.  Assume that
$\{B_{\delta}(x_{j})\}_{j=1}^{h}$ is a heavy maximal $\delta$-packing of $K$, and $\{\mathcal{E}_{2\delta}^{j}\}_{j=1}^{h}$ is the heavy maximal $4\delta$-partition of $K$ generated by $\{B_{\delta}(x_{j})\}_{j=1}^{h}$. For any $j\in\{1,\dots,h\}$, making use of (\ref{eq(2.06)}), we have ${\rm diam}(\mathcal{E}_{2\delta}^{j})\leq 4\delta$. For any $x\in \mathcal{E}_{2\delta}^{j}$, it follows that $\mathcal{E}_{2\delta}^{j}\subset B_{4\delta}(x)$. Hence for $0<q<1$,
\begin{equation}\label{eq(4.0)}
\mu(\mathcal{E}_{2\delta}^{j})^{q-1}\geq\mu(B_{4\delta}(x))^{q-1}.
\end{equation}
Thus,
\begin{equation}\label{eq(4.1)}
\begin{aligned}
\int_{K}\mu(B_{4\delta}(x))^{q-1}\,d\mu&=\sum_{j=1}^{h}\int_{\mathcal{E}_{2\delta}^{j}}\mu(B_{4\delta}(x))^{q-1}\,d\mu\\
&\leq\sum_{j=1}^{h}\int_{\mathcal{E}_{2\delta}^{j}}\mu(\mathcal{E}_{2\delta}^{j})^{q-1}\,d\mu\quad\text{(by (\ref{eq(4.0)}))}\\
&\leq C_{1}^{q}\sum_{j=1}^{h}\mu(B_{\delta}(x_{j}))^{q}\quad\text{(by (\ref{eq(2.04)}) and Proposition \ref{prop(2.1)})}.
\end{aligned}
\end{equation}
To prove the reverse inequality, we first use (\ref{eq(2.804)}) to get
$$B_{\delta}(x)\subset B_{3\delta}(x_{j})\subset\bigcup_{i:B_{3\delta}(x_{j})\cap \mathcal{E}_{2\delta}^{i}\neq\emptyset}\mathcal{E}_{2\delta}^{i}.$$
Let $\tilde{\mathcal{E}}_{2\delta}^{j}:=\bigcup_{i:i\neq j,B_{3\delta}(x_{j})\cap \mathcal{E}_{2\delta}^{i}\neq\emptyset}\mathcal{E}_{2\delta}^{i}$. Then
$B_{\delta}(x)\subset \mathcal{E}_{2\delta}^{j}\cup \tilde{\mathcal{E}}_{2\delta}^{j}$. Hence for $0<q<1$,
\begin{equation}\label{eq(4.00)}
\mu(B_{\delta}(x))^{q-1}\geq \big(\mu(\mathcal{E}_{2\delta}^{j})+\mu(\tilde{\mathcal{E}}_{2\delta}^{j})\big)^{q-1}.
\end{equation}
Thus,
\begin{equation}\label{eq(4.2)}
\begin{aligned}
\int_{K}\mu(B_{\delta}(x))^{q-1}\,d\mu&=\sum_{j=1}^{h}\int_{\mathcal{E}_{2\delta}^{j}}\mu(B_{\delta}(x))^{q-1}\,d\mu\\
&\geq\sum_{j=1}^{h}\int_{\mathcal{E}_{2\delta}^{j}}\big(\mu(\mathcal{E}_{2\delta}^{j})+\mu(\tilde{\mathcal{E}}_{2\delta}^{j})\big)^{q-1}\,d\mu\quad\text{(by (\ref{eq(4.00)}))}\\
&=\sum_{j=1}^{h}\frac{\mu(\mathcal{E}_{2\delta}^{j})}{\big(\mu(\mathcal{E}_{2\delta}^{j})+\mu(\tilde{\mathcal{E}}_{2\delta}^{j})\big)^{1-q}}.
\end{aligned}
\end{equation}
Let $L:=(2D_{0}5^{p})^{\frac{1}{q}}$, and
$$I:=\Bigg\{j:0\leq\frac{\mu(\tilde{\mathcal{E}}_{2\delta}^{j})}{L}<\mu(\mathcal{E}_{2\delta}^{j}),1\leq j\leq h\Bigg\},
\quad J:=\Bigg\{j:0<\mu(\mathcal{E}_{2\delta}^{j})\leq\frac{\mu(\tilde{\mathcal{E}}_{2\delta}^{j})}{L},1\leq j\leq h\Bigg\}.$$
Then
\begin{equation}\label{eq(4.3)}
\sum_{j=1}^{h}\frac{\mu(\mathcal{E}_{2\delta}^{j})}{\big(\mu(\mathcal{E}_{2\delta}^{j})+\mu(\tilde{\mathcal{E}}_{2\delta}^{j})\big)^{1-q}}
\geq\sum_{j\in I}\frac{\mu(\mathcal{E}_{2\delta}^{j})}{\big(\mu(\mathcal{E}_{2\delta}^{j})+\mu(\tilde{\mathcal{E}}_{2\delta}^{j})\big)^{1-q}}\geq(1+L)^{q-1}\sum_{j\in I}\mu(\mathcal{E}_{2\delta}^{j})^{q}.
\end{equation}
Lemma \ref{lem(2.2)} implies that for $0<q<1$,
\begin{equation}\label{eq(4.4)}
\sum_{j\in J}\mu(\mathcal{E}_{2\delta}^{j})^{q}\leq L^{-q}\sum_{j\in J}\mu(\tilde{\mathcal{E}}_{2\delta}^{j})^{q}
\leq L^{-q}\sum_{j\in J}\sum_{i:i\neq j,B_{3\delta}(x_{j})\cap \mathcal{E}_{2\delta}^{i}\neq\emptyset}\mu(\mathcal{E}_{2\delta}^{i})^{q}.
\end{equation}
If $B_{3\delta}(x_{j})\cap \mathcal{E}_{2\delta}^{i}\neq\emptyset$, then by (\ref{eq(2.803)}), $x_{j}\in B_{5\delta}(x_{i})$. By Lemma \ref{lem(2.1)}(c), the cardinality of any $\delta$-packing of
$B_{5\delta}(x_{i})$ is at most $D_{0}5^{p}$. Hence $\mathcal{E}_{2\delta}^{i}$ intersects at most $D_{0}5^{p}$ elements of $\{B_{3\delta}(x_{j})\}_{j=1}^{h}$. Thus, by (\ref{eq(4.4)}),
\begin{equation}\label{eq(4.41)}
\sum_{j\in J}\mu(\mathcal{E}_{2\delta}^{j})^{q}
\leq L^{-q}\sum_{j\in J}\sum_{i:B_{3\delta}(x_{j})\cap \mathcal{E}_{2\delta}^{i}\neq\emptyset}\mu(\mathcal{E}_{2\delta}^{i})^{q}
\leq D_{0}L^{-q}5^{p}\sum_{j=1}^{h}\mu(\mathcal{E}_{2\delta}^{j})^{q}=\frac{1}{2}\sum_{j=1}^{h}\mu(\mathcal{E}_{2\delta}^{j})^{q}.
\end{equation}
Making use of (\ref{eq(4.41)}), we have
\begin{equation}\label{eq(4.5)}
\sum_{j\in I}\mu(\mathcal{E}_{2\delta}^{j})^{q}
=\sum_{j=1}^{h}\mu(\mathcal{E}_{2\delta}^{j})^{q}-\sum_{j\in J}\mu(\mathcal{E}_{2\delta}^{j})^{q}\geq\frac{1}{2}\sum_{j=1}^{h}\mu(\mathcal{E}_{2\delta}^{j})^{q}.
\end{equation}
For any $j\in\{1,\dots,h\}$, we obtain from (\ref{eq(2.02)}) and (\ref{eq(2.04)}) that
\begin{equation}\label{eq(4.51)}
\mu(B_{\delta}(x_{j}))\leq\mu(\mathcal{E}_{2\delta}^{j}).
\end{equation}
It follows from (\ref{eq(4.2)}), (\ref{eq(4.3)}), (\ref{eq(4.5)}) and (\ref{eq(4.51)}) that
\begin{equation}\label{eq(4.6)}
\int_{K}\mu(B_{\delta}(x))^{q-1}\,d\mu\geq\frac{1}{2}(1+L)^{q-1}\sum_{j=1}^{h}\mu(\mathcal{E}_{2\delta}^{j})^{q}
\geq\frac{1}{2}(1+L)^{q-1}\sum_{j=1}^{h}\mu(B_{\delta}(x_{j}))^{q}.
\end{equation}
Combining (\ref{eq(4.1)}) and (\ref{eq(4.6)}) yields the theorem.
\end{proof}

\section{\bf Existence of the limit defining \boldmath$\tau(q)$ \label{sect.5}}

This section is devoted to the proof of Theorem \ref{thm(1.1)}.

\begin{proof}[Proof of Theorem \ref{thm(1.1)}]
Assume that $\{B_{2^{-t}}(x_{i})\}_{i=1}^{k}$ and $\{B_{2^{-s-t}}(y_{j})\}_{j=1}^{h}$ are heavy maximal $2^{-t}$ and
$2^{-s-t}$-packings of $K$ with $s,t\in\mathbb{N}$, respectively. Let $\{E_{2^{-t+1}}^{i}\}_{i=1}^{k}$ and $\{E_{2^{-s-t+1}}^{j}\}_{j=1}^{h}$ be defined with respect to $\{B_{2^{-t}}(x_{i})\}_{i=1}^{k}$ and $\{B_{2^{-s-t}}(y_{j})\}_{j=1}^{h}$ as in (\ref{eq:E_def}), respectively.
Combining Lemmas \ref{lem(2.3)} and \ref{lem(2.4)}, we obtain a constant $C_{3}\geq1$ independent of $s,t$ such that
\begin{equation}\label{eq(3.1)}
C_{3}^{-1}\mathcal{S}_{2^{-s}}(q)\leq\sum_{j=1}^{h}\mu(S_{u}^{-1}(E_{2^{-s-t+1}}^{j}))^{q}\leq C_{3}\mathcal{S}_{2^{-s}}(q)
\end{equation}
for any $t\in\mathbb{N}$, $s\geq N$ and $u\in \mathcal{W}_{t}$, where $N\in\mathbb{N}$. For convenience, write
$\mathcal{B}_{2^{-t}}:=\{B_{2^{-t}}(x_{i})\}_{i=1}^{k}$. For $B\in2\mathcal{B}_{2^{-t}}$, let $x$ be the center of $B$. For any $j\in\{1,\dots,h\}$, by (\ref{eq(2.02)}), we have
\begin{equation}\label{eq(3.111)}
E_{2^{-s-t+1}}^{j}\subset B_{2^{-s-t+1}}(y_{j})\subset B_{2^{-t}}(y_{j}).
\end{equation}
It follows from (\ref{eq(3.111)}) that if $E_{2^{-s-t+1}}^{j}\cap B\neq\emptyset$, then $B_{2^{-t}}(y_{j})\cap B\neq\emptyset$. Hence by (\ref{eq(2.801)}) and (\ref{eq(2.802)}), we have
\begin{equation}\label{eq(3.112)}
x\in B_{3\cdot2^{-t}}(y_{j})\quad\text{and}\quad E_{2^{-s-t+1}}^{j}\subset2*B.
\end{equation}
We divide the proof into two cases.

Case 1. $q\geq1$. We will show that there exists a constant $L\geq1$ such that
\begin{equation}\label{eq(3.2)}
\mathcal{S}_{2^{-s-t}}(q)\leq L\mathcal{S}_{2^{-s}}(q)\mathcal{S}_{2^{-t}}(q)
\end{equation}
for any $t\in\mathbb{N}$ and $s\geq N$, where $N\in\mathbb{N}$. That is, the sequence $\{L\mathcal{S}_{2^{-t}}(q)\}_{t\in\mathbb{N}}$ is sub-multiplicative. Hence the limit on the right side of (\ref{eq(2.8)}) exists, which means that the $L^{q}$-spectrum of $\mu$ exists.
Moreover, $\tau(q)=\sup_{t\geq N}\log(L\mathcal{S}_{2^{-t}}(q))/(-t\log2)$. We obtain from (\ref{eq(2.16)}) and (\ref{eq(3.112)}) that
\begin{equation}\label{eq(3.3)}
\mu(E_{2^{-s-t+1}}^{j})=\sum_{u\in \mathcal{W}_{t}:K_{u}\cap2*B\neq\emptyset}p_{u}\mu(S_{u}^{-1}(E_{2^{-s-t+1}}^{j})).
\end{equation}
For brevity, we will omit $u\in \mathcal{W}_{t}$ in the subscript for summation. Let
\begin{equation}\label{eq(3.31)}
P_{+}(B):=\sum_{u:K_{u}\cap2*B\neq\emptyset}p_{u}.
\end{equation}
Since $x^{q}$ is convex for $q\geq1$,
\begin{equation}\label{eq(3.4)}\begin{aligned}
\mu(E_{2^{-s-t+1}}^{j})^{q}
&=P_{+}(B)^{q}\Bigg(\sum_{u:K_{u}\cap2*B\neq\emptyset}\frac{p_{u}}{P_{+}(B)}\mu(S_{u}^{-1}(E_{2^{-s-t+1}}^{j}))\Bigg)^{q}\quad\text{(by (\ref{eq(3.3)}))}\\
&\leq P_{+}(B)^{q-1}\sum_{u:K_{u}\cap2*B\neq\emptyset}p_{u}\mu(S_{u}^{-1}(E_{2^{-s-t+1}}^{j}))^{q}\quad\text{(by (\ref{eq(3.31)}) and Jensen's inequality)}.
\end{aligned}\end{equation}
Summing over all $E_{2^{-s-t+1}}^{j}\cap B\neq\emptyset$, we have
\begin{equation}\label{eq(3.5)}\begin{aligned}
\sum_{j:E_{2^{-s-t+1}}^{j}\cap B\neq\emptyset}\mu(E_{2^{-s-t+1}}^{j})^{q}
&\leq P_{+}(B)^{q-1}\sum_{u:K_{u}\cap2*B\neq\emptyset} p_{u}
\sum_{j:E_{2^{-s-t+1}}^{j}\cap B\neq\emptyset}\mu(S_{u}^{-1}(E_{2^{-s-t+1}}^{j}))^{q}\quad\text{(by (\ref{eq(3.4)}))}\\
&\leq P_{+}(B)^{q-1}\sum_{u:K_{u}\cap2*B\neq\emptyset}p_{u}\cdot C_{3}\mathcal{S}_{2^{-s}}(q)\quad\text{(by (\ref{eq(3.1)}))}\\
&=C_{3}\mathcal{S}_{2^{-s}}(q)P_{+}(B)^{q}.
\end{aligned}\end{equation}
For any $j\in\{1,\dots,h\}$, by (\ref{eq(2.02)}), we have $B_{2^{-s-t}}(x_{j})\subset E_{2^{-s-t+1}}^{j}$. Summing over all $B\in2\mathcal{B}_{2^{-t}}$, it follows from (\ref{eq(3.5)}) that
\begin{equation}\label{eq(3.6)}
\sum_{j=1}^{h}\mu(B_{2^{-s-t}}(y_{j}))^{q}\leq \sum_{B\in2\mathcal{B}_{2^{-t}}}\sum_{j:E_{2^{-s-t+1}}^{j}\cap B\neq\emptyset}\mu(E_{2^{-s-t+1}}^{j})^{q}
\leq C_{3}\mathcal{S}_{2^{-s}}(q)\displaystyle{\sum_{B\in2\mathcal{B}_{2^{-t}}}}P_{+}(B)^{q}.
\end{equation}
Taking supremum on the left side of (\ref{eq(3.6)}), it follows from  Proposition \ref{prop(2.2)} that
\begin{equation}\label{eq(3.7)}
\mathcal{S}_{2^{-s-t}}(q)\leq C_{2}C_{3}\mathcal{S}_{2^{-s}}(q)\displaystyle{\sum_{B\in2\mathcal{B}_{2^{-t}}}}P_{+}(B)^{q}
:=c_{3}\mathcal{S}_{2^{-s}}(q)\displaystyle{\sum_{B\in2\mathcal{B}_{2^{-t}}}}P_{+}(B)^{q}.
\end{equation}
Note that ${\rm diam}(K_{u})\leq2^{-t}$ for any $u\in \mathcal{W}_{t}$. Hence for $B\in2\mathcal{B}_{2^{-t}}$, if $K_{u}\cap2*B\neq\emptyset$, then by (\ref{eq(2.802)}), $K_{u}\subset \frac{5}{2}*B$. For any $i\in\{1,\dots,k\}$, if $E_{2^{-t+1}}^{i}\cap \frac{5}{2}*B\neq\emptyset$, then it follows from (\ref{eq(2.803)}) that
\begin{equation}\label{eq(3.81)}
x_{i}\in \frac{7}{2}*B\quad\text{and}\quad x\in B_{7\cdot2^{-t}}(x_{i}),
\end{equation}
where $x$ is the center of $B$. By Lemma \ref{lem(2.1)}(c), the cardinality of any $2^{-t}$-packing of $\frac{7}{2}*B$ is at most $D_{0}7^{p}$. Hence $\frac{5}{2}*B$ intersects at most $D_{0}7^{p}$ elements of $\{E_{2^{-t+1}}^{i}\}_{i=1}^{k}$. Let
$A_{4}:=\max\{D_{0}^{q-1}7^{p(q-1)},1\}$. Since $\frac{5}{2}*B\subset\bigcup_{i:\frac{5}{2}*B\cap E_{2^{-t+1}}^{i}\neq\emptyset}E_{2^{-t+1}}^{i}$, it follows from Lemma \ref{lem(2.2)} that for $q>0$,
\begin{equation}\label{eq(3.8)}
P_{+}(B)^{q}\leq \Bigg(\sum_{i:\frac{5}{2}*B\cap E_{2^{-t+1}}^{i} \neq\emptyset}\mu(E_{2^{-t+1}}^{i})\Bigg)^{q}
\leq A_{4}\sum_{i:\frac{5}{2}*B\cap E_{2^{-t+1}}^{i} \neq\emptyset}\mu(E_{2^{-t+1}}^{i})^{q}.
\end{equation}
By Lemma \ref{lem(2.1)}(c), the cardinality of any $2^{-t}$-packing of $B_{7\cdot2^{-t}}(x_{i})$ is at most $D_{0}7^{p}$. In view of this and (\ref{eq(3.81)}), we see that $E_{2^{-t+1}}^{i}$ intersects at most $D_{0}7^{p}$ elements of $5\mathcal{B}_{2^{-t}}$. Thus,
\begin{equation}\label{eq(3.9)}\begin{aligned}
\sum_{B\in2\mathcal{B}_{2^{-t}}}P_{+}(B)^{q}
&\leq  A_{4}\sum_{B\in2\mathcal{B}_{2^{-t}}}\sum_{i:\frac{5}{2}*B\cap E_{2^{-t+1}}^{i}\neq\emptyset}\mu(E_{2^{-t+1}}^{i})^{q}
\quad\text{(by (\ref{eq(3.8)}))}\\
&\leq D_{0}A_{4}7^{p}\sum_{i=1}^{k}\mu(E_{2^{-t+1}}^{i})^{q}\\
&\leq D_{0}A_{4}C_{1}^{q}7^{p}\sum_{i=1}^{k}\mu(B_{2^{-t}}(x_{i}))^q\quad\text{(by Proposition \ref{prop(2.1)})}.
\end{aligned}\end{equation}
Taking supremum on the right side of (\ref{eq(3.9)}), we obtain from Proposition \ref{prop(2.2)} that
\begin{equation}\label{eq(3.10)}
\sum_{B\in2\mathcal{B}_{2^{-t}}}P_{+}(B)^{q}\leq D_{0}A_{4}C_{1}^{q}7^{p}\mathcal{S}_{2^{-t}}(q):=c_{4}\mathcal{S}_{2^{-t}}(q).
\end{equation}
Combining (\ref{eq(3.7)}) and (\ref{eq(3.10)}) yields (\ref{eq(3.2)}), where $L:=c_{3}c_{4}\geq1$.

Case 2. $0<q<1$. We will show that there exists a constant $L\geq1$ such that
\begin{equation}\label{eq(3.11)}
\mathcal{S}_{2^{-s-t}}(q)\geq L^{-1}\mathcal{S}_{2^{-s}}(q)\mathcal{S}_{2^{-t}}(q)
\end{equation}
for any $t\in\mathbb{N}$ and $s\geq N$, where $N\in\mathbb{N}$. That is, the sequence $\{L^{-1}\mathcal{S}_{2^{-t}}(q)\}_{t\in\mathbb{N}}$ is super-multiplicative. Hence the limit on the right side of (\ref{eq(2.8)}) exists, which means that the $L^{q}$-spectrum of $\mu$ exists. Moreover, $\tau(q)=\inf_{t\geq N}\log(L^{-1}\mathcal{S}_{2^{-t}}(q))/(-t\log2)$. For
$B\in2\mathcal{B}_{2^{-t}}$ and $u\in \mathcal{W}_{t}$, let
$$w(u,B):=\sum_{j:B\cap E_{2^{-s-t+1}}^{j}\neq\emptyset}\mu(S_{u}^{-1}(E_{2^{-s-t+1}}^{j}))^{q}.$$
Following \cite{Peres-Solomyak_2000}, we call a ball $B\in2\mathcal{B}_{2^{-t}}$ such that $w(u,B)$ attains its maximum \textit{$q$-heavy}. For each $u\in\mathcal W_t$, we fix a $q$-heavy ball and denote it by $H(u)$. Note that for any $u\in\mathcal{W}_{t}$, we have
$K_{u}\subset B_{2^{-t}}(y)$ for any $y\in K_{u}$. It follows from (\ref{eq(3.111)}) that if $E_{2^{-s-t+1}}^{j}\cap K_{u}\neq\emptyset$, then $B_{2^{-t}}(y_{j})\cap B_{2^{-t}}(y)\neq\emptyset$. Hence by (\ref{eq(2.801)}), we have $y_{j}\in B_{2^{-t+1}}(y)$.
By Lemma \ref{lem(2.1)}(c), the cardinality of any $2^{-t}$-packing of $B_{2^{-t+1}}(y)$ is at most $D_{0}2^{p}$.
Hence $K_{u}$ intersects at most $D_{0}2^{p}$ elements of $\{E_{2^{-s-t+1}}^{j}\}_{j=1}^{h}$. Combining this fact and (\ref{eq(3.1)}), we have for any $u\in\mathcal{W}_{t}$,
\begin{equation}\label{eq(3.12)}
D_{0}2^{p}\sum_{j:H(u)\cap E_{2^{-s-t+1}}^{j}\neq\emptyset}\mu(S_{u}^{-1}(E_{2^{-s-t+1}}^{j}))^{q}
\geq\sum_{j=1}^{h}\mu(S_{u}^{-1}(E_{2^{-s-t+1}}^{j}))^{q}\geq C_{3}^{-1}\mathcal{S}_{2^{-s}}(q).
\end{equation}
It follows from (\ref{eq(3.112)}) and (\ref{eq(3.3)}) that
\begin{equation}\label{eq(3.13)}
\mu(E_{2^{-s-t+1}}^{j})=\sum_{u:K_{u}\cap 2*B\neq\emptyset}p_{u}\mu(S_{u}^{-1}(E_{2^{-s-t+1}}^{j}))
\geq\sum_{u:B=H(u)}p_{u}\mu(S_{u}^{-1}(E_{2^{-s-t+1}}^{j})).
\end{equation}
Let
\begin{equation}\label{eq(3.131)}
P_{-}(B):=\sum_{u:B=H(u)}p_{u}.
\end{equation}
Observe that $x^{q}$ is concave for $0<q<1$. Hence
\begin{equation}\label{eq(3.14)}\begin{aligned}
\mu(E_{2^{-s-t+1}}^{j})^{q}&\geq P_{-}(B)^{q}\Bigg(\sum_{u:B=H(u)}\frac{p_{u}}{P_{-}(B)}\mu(S_{u}^{-1}(E_{2^{-s-t+1}}^{j}))\Bigg)^{q}
\quad\text{(by (\ref{eq(3.13)}))}\\
&\geq P_{-}(B)^{q-1}\displaystyle{\sum_{u:B=H(u)}}p_{u}\mu(S_{u}^{-1}(E_{2^{-s-t+1}}^{j}))^{q}
\quad\text{(by (\ref{eq(3.131)}) and Jensen's inequality)}.
\end{aligned}\end{equation}
Summing over all $ E_{2^{-s-t+1}}^{j}\cap B\neq\emptyset$, we have
\begin{equation}\label{eq(3.15)}
\begin{aligned}
\sum_{j:B\cap E_{2^{-s-t+1}}^{j}\neq\emptyset}\mu(E_{2^{-s-t+1}}^{j})^{q}
&\geq P_{-}(B)^{q-1}\sum_{u:B=H(u)}p_{u}\sum_{j:B\cap E_{2^{-s-t+1}}^{j}\neq\emptyset}\mu(S_{u}^{-1}(E_{2^{-s-t+1}}^{j}))^{q}
\quad\text{(by (\ref{eq(3.14)}))}\\
&\geq P_{-}(B)^{q}\cdot D_{0}^{-1}C_{3}^{-1}3^{-p}\mathcal{S}_{2^{-s}}(q)\quad\text{(by (\ref{eq(3.12)}))}.
\end{aligned}
\end{equation}
In view of (\ref{eq(3.112)}), if $B\cap E_{2^{-s-t+1}}^{j}\neq\emptyset$, then $x\in B_{3\cdot2^{-t}}(y_{j})$, where $x$ is the center of $B$.
By Lemma \ref{lem(2.1)}(c), the cardinality of any $2^{-t}$-packing of $B_{3\cdot2^{-t}}(y_{j})$ is at most $D_{0}3^{p}$. Hence $E_{2^{-s-t+1}}^{j}$ intersects at most $D_{0}3^{p}$ elements of $2\mathcal{B}_{2^{-t}}$. Summing over all $B\in2\mathcal{B}_{2^{-t}}$, we have
\begin{equation}\label{eq(3.16)}
\begin{aligned}
\sum_{B\in2\mathcal{B}_{2^{-t}}}\sum_{j:B\cap E_{2^{-s-t+1}}^{j}\neq\emptyset}\mu(E_{2^{-s-t+1}}^{j})^{q}&\leq D_{0}3^{p}\sum_{j=1}^{h}\mu(E_{2^{-s-t+1}}^{j})^{q}\\
&\leq D_{0}C_{1}^{q}3^{p}\sum_{j=1}^{h}\mu(B_{2^{-s-t}}(y_{j}))^{q}\quad\text{(by Proposition \ref{prop(2.1)})}\\
&\leq D_{0}C_{1}^{q}3^{p}\mathcal{S}_{2^{-s-t}}(q)\quad\text{(by Proposition \ref{prop(2.2)})}.
\end{aligned}
\end{equation}
We obtain from (\ref{eq(3.15)}) and (\ref{eq(3.16)}) that
\begin{equation}\label{eq(3.17)}
\mathcal{S}_{2^{-s-t}}(q)
\geq D_{0}^{-2}C_{3}^{-1}C_{1}^{-q}3^{-2p}\mathcal{S}_{2^{-s}}(q)\sum_{B\in2\mathcal{B}_{2^{-t}}}P_{-}(B)^{q}
:=c_{5}^{-1}\mathcal{S}_{2^{-s}}(q)\sum_{B\in2\mathcal{B}_{2^{-t}}}P_{-}(B)^{q}.
\end{equation}
We write $B\sim\tilde{B}$ if $\tilde{B}\in2\mathcal{B}_{2^{-t}}$ and $\tilde{B}\cap(2*B)\neq\emptyset$.
It follows from (\ref{eq(3.31)}) and (\ref{eq(3.131)}) that
$$\mu(B)\leq P_{+}(B)\leq\sum_{B\sim\tilde{B}}P_{-}(\tilde{B}).$$
Hence it follows from Lemma \ref{lem(2.2)} that for $0<q<1$,
$$\mu(B)^{q}\leq\Bigg(\sum_{B\sim\tilde{B}}P_{-}(\tilde{B})\Bigg)^{q}\leq\sum_{B\sim\tilde{B}}P_{-}(\tilde{B})^{q}.$$
Let $\tilde{x}$ be the center of $\tilde{B}$. If $\tilde{B}\cap2*B\neq\emptyset$, then we obtain from (\ref{eq(2.801)}) that $\tilde{x}\in3*B$. By Lemma \ref{lem(2.1)}(c), the cardinality of any $2^{-t}$-packing of $3*B$ is at most $D_{0}6^{p}$. Hence $2*B$ intersects at most $D_{0}6^{p}$ elements of $2\mathcal{B}_{2^{-t}}$. Summing over all $B\in2\mathcal{B}_{2^{-t}}$, we have
\begin{equation}\label{eq(3.18)}
\sum_{B\in2\mathcal{B}_{2^{-t}}}\mu\bigg(\frac{1}{2}*B\bigg)^{q}\leq\sum_{B\in2\mathcal{B}_{2^{-t}}}\mu(B)^{q}
\leq\sum_{B\in2\mathcal{B}_{2^{-t}}}\sum_{B\sim\tilde{B}}P_{-}(\tilde{B})^{q}
\leq D_{0}6^{p}\sum_{\tilde{B}\in2\mathcal{B}_{2^{-t}}}P_{-}(\tilde{B})^{q}.
\end{equation}
Taking supremum on the left side of (\ref{eq(3.18)}), we obtain from Proposition \ref{prop(2.2)} that
\begin{equation}\label{eq(3.19)}
\mathcal{S}_{2^{-t}}(q)\leq D_{0}C_{2}6^{p}\sum_{B\in2\mathcal{B}_{2^{-t}}}P_{-}(B)^{q}:=c_{6}\sum_{B\in2\mathcal{B}_{2^{-t}}}P_{-}(B)^{q}.
\end{equation}
Combining (\ref{eq(3.17)}) and (\ref{eq(3.19)}) yields (\ref{eq(3.11)}), where $L:=c_{5}c_{6}\geq1$.

Combining Cases 1 and 2 proves that for $q>0$, the limit defining $\tau(q)$ exists. It now follows directly from definition that for any $q\in(0,\infty)\setminus\{1\}$, the $L^{q}$-dimension of $\mu$ exists.
\end{proof}

\section{\bf Existence of the entropy dimension \label{sect.6}}

In this section, by assuming that $\mu$ is doubling, we extend the result concerning the existence of the entropy dimension in \cite{Peres-Solomyak_2000} to complete smooth Riemannian manifolds. We first give an equivalent definition of the entropy dimension before proving Theorem \ref{thm(1.2)}.

Let $\mu$ be a finite positive Borel measure on a Riemannian manifold with compact support $K$. Let
$\mathbb{E}_{2^{-t+2}}=\big\{\mathcal{E}_{2^{-t+1}}^{i}\big\}_{i=1}^{k}$ be a maximal $4(2^{-t})$-partition of $K$ with $t\in\mathbb{N}$.
Define
\begin{equation}\label{eq(5.001)}
h^{*}(\mu,\mathbb{E}_{2^{-t+2}}):=-\sum_{i=1}^{k}\mu(\mathcal{E}_{2^{-t+1}}^{i})\log\mu(\mathcal{E}_{2^{-t+1}}^{i}).
\end{equation}
For any $x\in K$, let $\mathcal{E}_{2^{-t+1}}^{i}(x)$ be the set in $\mathbb{E}_{2^{-t+2}}$ containing $x$, where $i=i_{x}\in\{1,\dots,k\}$. It follows from (\ref{eq(5.001)}) that
\begin{equation}\label{eq(5.00)}
h^{*}(\mu,\mathbb{E}_{2^{-t+2}})=-\sum_{i=1}^{k}\int_{\mathcal{E}_{2^{-t+1}}^{i}}\log\mu(\mathcal{E}_{2^{-t+1}}^{i})\,d\mu(x)
=-\int_{K}\log\mu(\mathcal{E}_{2^{-t+1}}^{i}(x))\,d\mu(x).
\end{equation}
Let
\begin{equation}\label{eq(5.01)}
h^{*}_{t}(\mu):=\inf\Big\{h^{*}(\mu,\mathbb{E}_{2^{-t+2}}):\mathbb{E}_{2^{-t+2}}~\text{is a maximal}~4(2^{-t})\text{-partition of}~K\Big\}.
\end{equation}

\begin{prop}\label{prop(5.1)}
Let $h^{*}_{t}(\mu)$ be defined as in (\ref{eq(5.01)}). Then the limit defining entropy dimension in (\ref{eq:edim_def}) exists if and only if $\lim_{t\rightarrow\infty}h_{t}^{*}(\mu)/(t\log2)$ exists. Moreover, if either limit exists, then
\begin{equation}\label{eq(5.6)}
\dim_{e}(\mu)=\lim_{t\rightarrow\infty}\frac{h_{t}^{*}(\mu)}{t\log2}.
\end{equation}
\end{prop}

\begin{proof}
Assume that $\mathcal{P}$ is a finite Borel $\delta$-partition of $K$, and the entropy dimension of $\mu$ is defined as in (\ref{eq:edim_def}). For any $\delta>0$, there exists a unique $t=t(\delta)\in\mathbb{N}$ such that $2^{-t-1}\leq\delta<2^{-t}$.
For any $\mathcal{E}_{2^{-t+1}}^{i}\in \mathbb{E}_{2^{-t+2}}$, where $i\in\{1,\dots,k\}$, we have by (\ref{eq(2.06)}) that ${\rm diam}(\mathcal{E}_{2^{-t+1}}^{i})\leq2^{-t+2}\leq8\delta$.
It follows from definition that
\begin{equation}\label{eq(5.1)}
h(\mu,8\delta)\leq h^{*}_{t}(\mu).
\end{equation}

On the other hand, for any $x\in K$, let $P(x)$ be the set in $\mathcal{P}$ containing $x$. According to (\ref{eq(5.00)}),
\begin{equation}\label{eq(5.4)}
h(\mu,\mathcal{P})=-\int_{K}\log\mu(P(x))\,d\mu(x).
\end{equation}
For any $x\in\mathcal{E}_{2^{-t+1}}^{i}$, note that $P(x)\cap \mathcal{E}_{2^{-t+1}}^{i}\neq\emptyset$. Since ${\rm diam}(P(x))\leq\delta<2^{-t}$, it follows from (\ref{eq(2.804)}) that $P(x)\subset B_{3\cdot2^{-t}}(x_{i})$.
Let $C\geq1$ be a doubling constant of $\mu$. It follows from (\ref{eq(1.1)}), (\ref{eq(2.02)}) and (\ref{eq(2.04)}) that
\begin{equation}\label{eq(5.2)}
\mu(P(x))\leq\mu(B_{3\cdot2^{-t}}(x_{i}))\leq C^{\log_{2}\frac{3}{2}}\mu(B_{2^{-t}}(x_{i}))\leq
C^{\log_{2}\frac{3}{2}}\mu(E_{2^{-t+1}}^{i})=C^{\log_{2}\frac{3}{2}}\mu(\mathcal{E}_{2^{-t+1}}^{i}).
\end{equation}
Thus,
\begin{equation}\label{eq(5.3)}\begin{aligned}
\int_{K}\log\mu(P(x))\,d\mu(x)&=\sum_{i=1}^{k}\int_{\mathcal{E}_{2^{-t+1}}^{i}}\log\mu(P(x))\,d\mu(x)\\
&\leq\sum_{i=1}^{k}\int_{\mathcal{E}_{2^{-t+1}}^{i}}\big(\log C^{\log_{2}\frac{3}{2}}+\log\mu(\mathcal{E}_{2^{-t+1}}^{i})\big)\,d\mu(x)\quad\text{(by~(\ref{eq(5.2)}))}\\
&\leq (\log_{2}3-1)\log C+\sum_{i=1}^{k}\mu(\mathcal{E}_{2^{-t+1}}^{i})\log\mu(\mathcal{E}_{2^{-t+1}}^{i}).
\end{aligned}\end{equation}
It follows from (\ref{eq(5.001)}), (\ref{eq(5.01)}), (\ref{eq(5.4)}) and (\ref{eq(5.3)}) that
\begin{equation}\label{eq(5.5)}
h(\mu,\delta)\geq(1-\log_{2}3)\log C+h_{t}^{*}(\mu).
\end{equation}
The proposition now follows by combining (\ref{eq(5.1)}) and (\ref{eq(5.5)}).
\end{proof}

\begin{lem}\label{lem(5.1)}
Let $h^{*}_{t}(\mu)$ be defined as in (\ref{eq(5.01)}). Then\\
(a) for any $t\in\mathbb{N}$, $h_{t+1}^{*}(\mu)\leq h_{t}^{*}(\mu)+C^{\log_{2}10}$, where $C\geq1$ is a doubling constant of $\mu$;\\
(b) there exists a constant $C_{4}>0$ such that for any $t\in\mathbb{N}$,
$$\left|h_{t}^{*}(\mu)+\int_{K}\log\mu(B_{2^{-t}}(x))\,d\mu\right|\leq C_{4}.$$
\end{lem}

\begin{proof}
Let $\mathbb{E}_{2^{-t+2}}=\{\mathcal{E}_{2^{-t+1}}^{i}\}_{i=1}^{k}$ and
$\mathbb{E}_{2^{-t+1}}=\{\mathcal{E}_{2^{-t}}^{j}\}_{j=1}^{h}$ be arbitrary maximal $4(2^{-t})$ and $4(2^{-t-1})$-partitions of $K$ with $t\in\mathbb{N}$, respectively. For any $x\in K$, let $\mathcal{E}_{2^{-t+1}}^{i}(x)$ and $\mathcal{E}_{2^{-t}}^{j}(x)$ be the sets containing $x$, where $i=i_{x}\in\{1,\dots,k\}$ and $j=j_{x}\in\{1,\dots,h\}$.

(a) Making use of (\ref{eq(2.06)}), we have ${\rm diam}(\mathcal{E}_{2^{-t+1}}^{i})\leq2^{-t+2}$. It follows from (\ref{eq(2.804)}) that if $\mathcal{E}_{2^{-t+1}}^{i}\cap\mathcal{E}_{2^{-t}}^{j}\neq\emptyset$, then $\mathcal{E}_{2^{-t+1}}^{i}\subset B_{5\cdot2^{-t}}(x_{j})$. Hence for any $x\in \mathcal{E}_{2^{-t}}^{j}$, we have
$$\mathcal{E}_{2^{-t+1}}^{i}(x)\subset B_{5\cdot2^{-t}}(x_{j}).$$
Combining this with (\ref{eq(1.1)}), (\ref{eq(2.02)}) and (\ref{eq(2.04)}), we get
\begin{equation}\label{eq(5.7)}
\mu(\mathcal{E}_{2^{-t+1}}^{i}(x))\leq\mu(B_{5\cdot2^{-t}}(x_{j}))\leq C^{\log_{2}10}\mu(B_{2^{-t-1}}(x_{j}))\leq C^{\log_{2}10}\mu(E_{2^{-t}}^{j})
=C^{\log_{2}10}\mu(\mathcal{E}_{2^{-t}}^{j}).
\end{equation}
Thus, for any $t\in\mathbb{N}$,
\begin{equation}\label{eq(5.70)}\begin{aligned}
h^{*}(\mu,\mathbb{E}_{2^{-t+1}})-h^{*}(\mu,\mathbb{E}_{2^{-t+2}})
&=\sum_{j=1}^{h}\int_{\mathcal{E}_{2^{-t}}^{j}}\log\frac{\mu(\mathcal{E}_{2^{-t+1}}^{i}(x))}{\mu(\mathcal{E}_{2^{-t}}^{j})}\,d\mu(x)\quad\text{(by (\ref{eq(5.00)}))}\\
&\leq\sum_{j=1}^{h}\int_{\mathcal{E}_{2^{-t}}^{j}}\frac{\mu(\mathcal{E}_{2^{-t+1}}^{i}(x))}{\mu(\mathcal{E}_{2^{-t}}^{j})}\,d\mu(x)\\
&\leq C^{\log_{2}10}\quad\text{(by (\ref{eq(5.7)}))}.
\end{aligned}\end{equation}
It follows from (\ref{eq(5.70)}) that
\begin{equation}\label{eq(5.701)}
h^{*}(\mu,\mathbb{E}_{2^{-t+1}})\leq h^{*}(\mu,\mathbb{E}_{2^{-t+2}})+ C^{\log_{2}10}.
\end{equation}
The result in (a) follows by (\ref{eq(5.701)}) and (\ref{eq(5.01)}).

(b) For any $x\in K$, by (\ref{eq(2.06)}), we have ${\rm diam}(\mathcal{E}_{2^{-t+1}}^{i}(x))\leq2^{-t+2}$.
Hence $\mathcal{E}_{2^{-t+1}}^{i}(x)\subset B_{2^{-t+2}}(x)$. Combining this and (\ref{eq(5.00)}) for any $t\in\mathbb{N}$,
\begin{equation}\label{eq(5.71)}
h^{*}(\mu,\mathbb{E}_{2^{-t+2}})=-\int_{K}\log\mu(\mathcal{E}_{2^{-t+1}}^{i}(x))\,d\mu(x)\geq-\int_{K}\log\mu(B_{2^{-t+2}}(x))\,d\mu(x).
\end{equation}
For any integer $m\geq2$, we have
\begin{equation}\label{eq(5.72)}\begin{aligned}
h^{*}(\mu,\mathbb{E}_{2^{-t+2}})&\geq h^{*}(\mu,\mathbb{E}_{2^{-t-m+2}})-mC^{\log_{2}10}\quad\text{(by~(a))}\\
&\geq-\int_{K}\log\mu(B_{2^{-t-m+2}}(x))\,d\mu-mC^{\log_{2}10}\quad\text{(by~(\ref{eq(5.71)}))}\\
&\geq-\int_{K}\log\mu(B_{2^{-t}}(x))\,d\mu-mC^{\log_{2}10}.
\end{aligned}\end{equation}

On the other hand, for any $x\in\mathcal{E}_{2^{-t+1}}^{i}$, it follows from (\ref{eq(2.804)}) that $B_{2^{-t}}(x)\subset B_{3\cdot2^{-t}}(x_{i})$.
Making use of (\ref{eq(1.1)}), (\ref{eq(2.02)}) and (\ref{eq(2.04)}), we have
\begin{equation}\label{eq(5.8)}
\mu(B_{2^{-t}}(x))\leq\mu(B_{2^{-t+2}}(x_{i}))\leq C^{\log_{2}3}\mu(B_{2^{-t}}(x_{i}))\leq C^{\log_{2}3}\mu(E_{2^{-t+1}}^{i})
=C^{\log_{2}3}\mu(\mathcal{E}_{2^{-t+1}}^{i}).
\end{equation}
Thus,
\begin{equation}\label{eq(5.81)}\begin{aligned}
h^{*}(\mu,\mathbb{E}_{2^{-t+2}})+\int_{K}\log\mu(B_{2^{-t}}(x))\,d\mu
&=\sum_{i=1}^{k}\int_{\mathcal{E}_{2^{-t+1}}^{i}}\log\frac{\mu(B_{2^{-t}}(x))}{\mu(\mathcal{E}_{2^{-t+1}}^{i})}\,d\mu\quad\text{(by (\ref{eq(5.00)}))}\\
&\leq\sum_{i=1}^{k}\int_{\mathcal{E}_{2^{-t+1}}^{i}}\frac{\mu(B_{2^{-t}}(x))}{\mu(\mathcal{E}_{2^{-t+1}}^{i})}\,d\mu\\
&\leq C^{\log_{2}3}\quad\text{(by~(\ref{eq(5.8)}))}.
\end{aligned}\end{equation}
The result in (b) follows by taking infimum on the left sides of (\ref{eq(5.72)}) and (\ref{eq(5.81)}), and using (\ref{eq(5.01)}), where $C_{4}:=\max\{C^{\log_{2}3},mC^{\log_{2}10}\}>0$.
\end{proof}

\begin{proof}[Proof of Theorem \ref{thm(1.2)}]
Let $\mathbb{E}_{2^{-t+2}}=\{\mathcal{E}_{2^{-t+1}}^{i}\}_{i=1}^{k}$ and $\mathbb{E}_{2^{-s-t+2}}=\{\mathcal{E}_{2^{-s-t+1}}^{j}\}_{j=1}^{h}$ be arbitrary maximal $4(2^{-t})$ and $4(2^{-s-t})$-partitions of $K$ with $s,t\in\mathbb{N}$, respectively.
We will show that there exists a constant $L>0$ such that
\begin{equation}\label{eq(5.90)}
h_{s+t}^{*}(\mu)\geq h_{s}^{*}(\mu)+h_{t}^{*}(\mu)-L,
\end{equation}
for any $t\in\mathbb{N}$ and $s\geq N$, where $N\in\mathbb{N}$.
That is, the sequence $\{h_{t}^{*}(\mu)-L\}_{t\in\mathbb{N}}$ is super-additive. Hence the limit on the right side of (\ref{eq(5.6)}) exists.
By Proposition \ref{prop(5.1)}, the entropy dimension of $\mu$ exists.
Moreover, $\dim_{e}(\mu)=\sup_{t\geq N}\displaystyle{\frac{h_{t}^{*}(\mu)-L}{t\log2}}$. For $E\in \mathbb{E}_{2^{-s-t+2}}$, let
\begin{equation}\label{eq(5.91)}
P(E):=\sum_{u\in\mathcal{W}_{t}:K_{u}\cap E\neq\emptyset}p_{u}.
\end{equation}
For brevity, we will omit $u\in \mathcal{W}_{t}$ in the subscript for summation.
Let $f(x):=-x\log x$. Then $f(xy)=f(x)y+xf(y)$ for any $x,y>0$. Letting $x=P(E)$ and $y=\mu(E)/P(E)$, we have
$$\begin{aligned}
f(\mu(E))&=f(P(E))\frac{\mu(E)}{P(E)}+P(E)f\bigg(\frac{\mu(E)}{P(E)}\bigg)\\
&=f\big(P(E)\big)\frac{\mu(E)}{P(E)}+P(E)f\Bigg(\sum_{u:K_{u}\cap E\neq\emptyset}\frac{p_{u}}{P(E)}\mu(S_{u}^{-1}(E))\Bigg)\quad\text{(by (\ref{eq(2.16)}))}\\
&=:f_{1}(E)+f_{2}(E).
\end{aligned}$$
Hence by (\ref{eq(5.001)}),
\begin{equation}\label{eq(5.9)}
h^{*}(\mu,\mathbb{E}_{2^{-s-t+2}})=\sum_{E\in\mathbb{E}_{2^{-s-t+2}}}f(\mu(E))=\sum_{E\in\mathbb{E}_{2^{-s-t+2}}}\big(f_{1}(E)+f_{2}(E)\big).
\end{equation}
Observe that
\begin{equation}\label{eq(5.901)}
f_{1}(E)=-\mu(E)\log P(E).
\end{equation}
Making use of (\ref{eq(2.06)}), we have ${\rm diam}(E)\leq 2^{-s-t+2}\leq2^{-t+1}$. It follows that $E\subset B_{2^{-t+1}}(x)$ for any $x\in E$. Recall that ${\rm diam}(K_{u})\leq2^{-t}$ for any $u\in\mathcal{W}_{t}$. Hence if $K_{u}\cap E\neq\emptyset$, then $K_{u}\cap B_{2^{-t+1}}(x)\neq\emptyset$. Thus, by (\ref{eq(2.802)}), $K_{u}\subset B_{3\cdot2^{-t}}(x)\subset B_{2^{-t+2}}(x)$. It follows from (\ref{eq(5.91)}) that
\begin{equation}\label{eq(5.92)}
P(E)\leq\mu(B_{2^{-t+2}}(x)).
\end{equation}
Summing over all $E\in\mathbb{E}_{2^{-s-t+2}}$, we get
\begin{equation}\label{eq(5.10)}\begin{aligned}
\sum_{E\in\mathbb{E}_{2^{-s-t+2}}}f_{1}(E)
&=-\sum_{E\in\mathbb{E}_{2^{-s-t+2}}}\int_{E}\log P(E)\,d\mu\quad\text{(by (\ref{eq(5.901)}))}\\
&\geq-\int_{K}\log\mu(B_{2^{-t+2}}(x))\,d\mu\quad\text{(by (\ref{eq(5.92)}))}\\
&\geq h_{t-2}^{*}(\mu)-C_{4}\quad\text{(by Lemma \ref{lem(5.1)}(b))}\\
&\geq h_{t}^{*}(\mu)-C_{4}-2C^{\log_{2}10}\quad\text{(by Lemma \ref{lem(5.1)}(a))}.
\end{aligned}\end{equation}
Since $f$ is concave, we have
\begin{equation}\label{eq(5.12)}
f_{2}(E)\geq P(E)\sum_{u:K_{u}\cap E\neq\emptyset}\frac{p_{u}}{P(E)}f(\mu(S_{u}^{-1}(E)))
=\sum_{u:K_{u}\cap E\neq\emptyset}p_{u}f(\mu(S_{u}^{-1}(E))).
\end{equation}
Summing over all $E\in\mathbb{E}_{2^{-s-t+2}}$, and using (\ref{eq(5.12)}), we get
\begin{equation}\label{eq(5.121)}
\sum_{E\in\mathbb{E}_{2^{-s-t+2}}}f_{2}(E)\geq
\sum_{E\in\mathbb{E}_{2^{-s-t+2}}}\sum_{u:K_{u}\cap E\neq\emptyset}p_{u}f(\mu(S_{u}^{-1}(E)))
=\sum_{u\in\mathcal{W}_{t}}p_{u}\sum_{E:K_{u}\cap E\neq\emptyset}f(\mu(S_{u}^{-1}(E))).
\end{equation}
By Lemma \ref{lem(2.4)}, $\{S_{u}^{-1}(E)\cap K\}_{E\in\mathbb{E}_{2^{-s-t+2}}}$ is a $(Q,2^{-s},D)$-good cover of $K$ for any $s\geq N$, where $Q>0$ and $D,N\in\mathbb{N}$.
Hence ${\rm diam}(S_{u}^{-1}(E))\leq Q\cdot2^{-s}$. Consequently, for any $u\in\mathcal{W}_{t}$ and $x\in S_{u}^{-1}(E)$,
\begin{equation}\label{eq(5.13)}
S_{u}^{-1}(E)\subset B_{Q\cdot2^{-s}}(x).
\end{equation}
Thus,
\begin{equation}\label{eq(5.14)}\begin{aligned}
\sum_{E:K_{u}\cap E\neq\emptyset}f(\mu(S_{u}^{-1}(E)))&=-\sum_{E:K_{u}\cap E\neq\emptyset}\mu(S_{u}^{-1}(E))\log\mu(S_{u}^{-1}(E))\\
&\geq-\sum_{E:K_{u}\cap E\neq\emptyset}\int_{S_{u}^{-1}(E)}\log\mu(B_{Q\cdot2^{-s}}(x))\,d\mu\quad\text{(by (\ref{eq(5.13)}))}\\
&\geq-\int_{K}\log\mu(B_{Q\cdot2^{-s}}(x))\,d\mu.
\end{aligned}\end{equation}
For $Q>0$, there exists an integer $m\geq2$ such that $Q\leq2^{m}$. It follows from (\ref{eq(5.14)}) and Lemma \ref{lem(5.1)} that
\begin{equation}\label{eq(5.15)}
\sum_{E:K_{u}\cap E\neq\emptyset}f\big(\mu(S_{u}^{-1}(E))\big)\geq-\int_{K}\log\mu(B_{2^{m-s}}(x))\,d\mu
\geq h_{s-m}^{*}(\mu)-C_{4}\geq h_{s}^{*}(\mu)-C_{4}-mC^{\log_{2}10}.
\end{equation}
Since $\sum_{u\in\mathcal{W}_{t}}p_{u}=1$, by (\ref{eq(5.121)}) and (\ref{eq(5.15)}), we have
\begin{equation}\label{eq(5.16)}
\sum_{E\in\mathbb{E}_{2^{-s-t+2}}}f_{2}(E)
\geq h_{s}^{*}(\mu)-C_{4}-mC^{\log_{2}10}.
\end{equation}
Taking the infimum on the left side of (\ref{eq(5.9)}), and
combining (\ref{eq(5.10)}) and (\ref{eq(5.16)}), we obtain (\ref{eq(5.90)}) with $L:=2C_{4}+(m+2)C^{\log_{2}10}>0$.
\end{proof}

\section{\bf Examples \label{sect.7}}

In this section, we give some examples of self-conformal measures and doubling self-conformal measures defined on Riemannian manifolds.

Let $\{f_{i}\}_{i=1}^{\ell}$ be a CIFS on a compact set $W_0\subset\mathbb{R}^{n}$, i.e., $f_{i}$ is $C^{1+\gamma}$ and there exists an open and connected set $U_{0}\supset W_{0}$ such that for any $i\in\{1,\dots,\ell\}$, $f_{i}$ can be extended to an injective conformal map $f_{i}:U_{0}\rightarrow U_{0}$. Let $M$ be a complete $n$-dimensional smooth Riemannian manifold. Then there exists a diffeomorphism $\varphi:U\rightarrow U_{0}$, where $U\subset M$ is open and connected.
Let $\mu_{0}$ be the unique self-conformal measure with compact support satisfying
\begin{equation}\label{eq(6.0)}
\mu_{0}=\sum_{i=1}^{\ell}p_{i}\mu_{0}\circ f_{i}^{-1}.
\end{equation}
Define
\begin{equation}\label{eq(6.1)}
S_{i}:=\varphi^{-1}\circ f_{i}\circ\varphi:U\rightarrow S(U)\quad\text{for any }i\in\{1,\dots,\ell\},
\end{equation}
and
\begin{equation}\label{eq(6.2)}
\mu:=\mu_{0}\circ\varphi.
\end{equation}

Then we have the following proposition.

\begin{prop}\label{prop(6.1)}
Let $M$ be a complete $n$-dimensional smooth Riemannian manifold, and $U\subset M$ be open and connected. Assume that $\{S_{i}\}_{i=1}^{\ell}$ and $\mu$ are defined as in (\ref{eq(6.1)}) and (\ref{eq(6.2)}). Then $\{S_{i}\}_{i=1}^{\ell}$ is a CIFS on $U$ and there exists a compact set $W\subset U$ such that $S_{i}(W)\subset W$ for any $i\in\{1,\dots,\ell\}$. Moreover, $\mu$ is the self-conformal measure generated by $\{S_{i}\}_{i=1}^{\ell}$.
\end{prop}

\begin{proof}
For any $i\in\{1,\dots,\ell\}$, making use of (\ref{eq(6.1)}), we have
\begin{equation}\label{eq(6.20)}
S_{i}(U)=\varphi^{-1}\circ f_{i}\circ\varphi(U)=\varphi^{-1}\circ f_{i}(U_{0})\subseteq\varphi^{-1}(U_{0})=U.
\end{equation}
Similar to (\ref{eq(6.20)}), there exists a compact set $W\subset U$ such that $\varphi(W)=W_{0}$ and $S_{i}(W)\subset W$ for any $i\in\{1,\dots,\ell\}$. Since $\{f_{i}\}_{i=1}^{\ell}$ is a CIFS, it follows from the equality $S'_{i}=f'_{i}$ that  $S_{i}$ is a conformal $C^{1+\gamma}$ diffeomorphism on $U$ and $0<|\det S'_{i}(x)|<1$ for any $x\in U$ and $i\in\{1,\dots,\ell\}$. Hence $\{S_{i}\}_{i=1}^{\ell}$ is a CIFS on $U$.
It follows from (\ref{eq(6.0)}), (\ref{eq(6.1)}) and (\ref{eq(6.2)}) that
$$\mu=\sum_{i=1}^{\ell}p_{i}\mu_{0}\circ f_{i}^{-1}\circ\varphi
=\sum_{i=1}^{\ell}p_{i}\mu_{0}\circ\varphi\circ\varphi^{-1}\circ f_{i}^{-1}\circ\varphi=\sum_{i=1}^{\ell}p_{i}\mu\circ S_{i}^{-1}.$$
The proposition follows by the uniqueness of self-conformal measure (see \cite{Hutchinson_1981}).
\end{proof}

Suppose $\mu_0$ is doubling. In the following proposition, we construct a class of examples of doubling self-conformal measures on Riemannian manifolds satisfying the hypotheses of Theorem \ref{thm(1.2)}.

\begin{prop}\label{prop(6.2)}
Assume the same hypotheses of Proposition \ref{prop(6.1)}, and assume in addition that $\mu_{0}$ is doubling.
Suppose that for any ball $B_{r}(x)\subset M$, there exist two positive constants $d_{1}$ and $d_{2}$ such that
\begin{equation}\label{eq(6.21)}
B_{r_{1}}(\varphi(x))\subset\varphi(B_{r}(x))\subset B_{r_{2}}(\varphi(x)),
\end{equation}
where $r_{1}$ and $r_{2}$ satisfy
$$d_{1}\leq\frac{r_1}{r}\leq\frac{r_2}{r}\leq d_{2}.$$
Then $\mu$ is doubling.
\end{prop}

\begin{proof}
Let $C_{0}$ be a doubling constant of $\mu_{0}$. Note that there exists $m\in\mathbb{N}$ such that
\begin{equation}\label{eq(6.3)}
r_{2}\leq\frac{d_{2}}{d_{1}}r_{1}\leq2^{m}r_{1}.
\end{equation}
Since $\mu_{0}$ is doubling, for any $B_{r}(x)\subset M$ we have
$$\begin{aligned}
\mu(B_{2r}(x))&=\mu_{0}\circ\varphi(B_{2r}(x))\quad\text{(by (\ref{eq(6.2)}))}\\
&\leq\mu_{0}(B_{2r_{2}}(\varphi(x)))\quad\text{(by (\ref{eq(6.21)}))}\\
&\leq\mu_{0}(B_{2^{m+1}r_{1}}(\varphi(x)))\quad\text{(by (\ref{eq(6.3)}))}\\
&\leq C_{0}^{m+1}\mu_{0}(B_{r_{1}}(\varphi(x)))\\
&\leq C_{0}^{m+1}\mu(B_{r}(x))\quad\text{(by (\ref{eq(6.21)}))}.
\end{aligned}$$
The proposition follows.
\end{proof}

The following example is an actual case of Proposition \ref{prop(6.2)}.

\begin{exam}\label{exam(6.1)}
Let
$$\mathbb{S}^{n}:=\bigg\{(x_{1},\dots,x_{n+1})\in \mathbb{R}^{n+1}:\sum_{i=1}^{n+1}x_{i}^{2}=1\bigg\}$$
and
$$\mathbb{D}^{n}:=\bigg\{(x_{1},\dots,x_{n})\in \mathbb{R}^{n}:\sum_{i=1}^{n}x_{i}^{2}<1\bigg\}.$$
Let $\tilde{\mathbb{S}}^{n}$ be the lower hemisphere of $\mathbb{S}^{n}$ and  define the stereographic projection $\varphi:\tilde{\mathbb{S}}^{n}\rightarrow \mathbb{D}^{n}$ as
$$\varphi(x_{1},\dots,x_{n+1})=\frac{1}{1-x_{n+1}}(x_{1},\dots,x_{n}):=(u_{1},\dots,u_{n}).$$
Then
$$\varphi^{-1}(u_{1},\dots,u_{n})=\frac{1}{|\boldsymbol{u}|^{2}+1}(2u_{1},\dots,2u_{n},|\boldsymbol{u}|^{2}-1),$$
where $|\boldsymbol{u}|^{2}=u_{1}^{2}+\cdots+u_{n}^{2}$. Let $\{f_{i}\}_{i=1}^{\ell}$ be a CIFS with compact support on $\mathbb{D}^{n}$ and $S_{i}:=\varphi^{-1}\circ f_{i}\circ\varphi$ for any $i\in\{1,\dots,\ell\}$. By Proposition \ref{prop(6.1)}, $\{S_{i}\}_{i=1}^{\ell}$ is a CIFS with compact support on
$\tilde{\mathbb{S}}^{n}$. It is well known that $\varphi$ is conformal. Let $B_{r}(x)\subset\tilde{\mathbb{S}}^{n}$ and $B_{r'}(\varphi(x))$ be its image on $\mathbb{D}^{n}$. Then a simple calculation shows that $1<r/r'<2$.
\end{exam}

\begin{rem}
One can use Proposition \ref{prop(6.1)} to construct CIFSs and self-conformal measures on Riemannian manifolds by making use of CIFSs on $\mathbb{R}^{n}$ (see \cite{Lau-Ngai-Wang_2009,Ye_2005}).
\end{rem}

\begin{appendix}\section{}

For completeness, we include the proof that each metric measure space carrying a doubling measure must be doubling, and the proof of Lemma \ref{lem(2.2)}.

\begin{prop}\label{A_1}
Let $(X,\rho,\mu)$ be a metric measure space and $\mu$ be doubling. Then $X$ is doubling.
\end{prop}

\begin{proof}
For any $x\in X$, let $\mathcal{B}$ be a maximal $(r/2)$-packing of $B_{2r}(x)$. Then $2\mathcal{B}$ covers $B_{2r}(x)$ and $\bigcup_{B\in\mathcal{B}}B\subset B_{5r/2}(x)$. We denote the ball in $\mathcal{B}$ whose measure is the smallest by $B_{r/2}(\tilde{x})$, where $\tilde{x}\in B_{2r}(x)$. If there are more than one such
ball, we choose one arbitrarily. It follows that $d(x,\tilde{x})\leq2r$. Thus,
$B_{5r/2}(x)\subset B_{9r/2}(\tilde{x})$. Let $C$ be a doubling constant of $\mu$. Then
$$\#\mathcal{B}\cdot\mu(B_{\frac{r}{2}}(\tilde{x}))\leq\sum_{B\in\mathcal{B}}\mu(B)
\leq\mu(B_{\frac{5}{2}r}(x))\leq\mu(B_{\frac{9}{2}r}(\tilde{x}))\leq C^{9}\mu(B_{\frac{r}{2}}(\tilde{x})),$$
Hence $\#\mathcal{B}\leq C^{9}$, proving that $X$ is doubling.
\end{proof}

\begin{proof}[Proof of Theorem \ref{lem(2.2)}]
We assume $\sum_{i=1}^{k}a_{i}>0$; otherwise the inequality holds trivially. For $0<q\leq1$, since $\displaystyle{\frac{a_{i}}{\sum_{i=1}^{k}a_{i}}}\leq1$, we have $\displaystyle{\Bigg(\frac{a_{i}}{\sum_{i=1}^{k}a_{i}}\Bigg)^{q}}\geq\displaystyle{\frac{a_{i}}{\sum_{i=1}^{k}a_{i}}}$, where $i\in\{1,\dots,k\}$. Summing over all
$i\in\{1,\dots,k\}$, we have
$$\sum_{i=1}^{k}\frac{{a_{i}}^{q}}{\big(\sum_{i=1}^{k}a_{i}\big)^{q}}\geq\sum_{i=1}^{k}\frac{a_{i}}{\sum_{i=1}^{k}a_{i}}=1.$$
Hence $\big(\sum_{i=1}^{k}a_{i}\big)^{q}\leq\sum_{i=1}^{k}a_{i}^{q}$. For $q>1$, it follows from H\"{o}lder's inequality that
$$\Bigg(\sum_{i=1}^{k}a_{i}\Bigg)^{q}\leq k^{q-1}\sum_{i=1}^{k}a_{i}^{q}.$$
The asserted inequality follows.
\end{proof}

\end{appendix}

\end{document}